\documentclass[11pt]{article}
\usepackage[a4paper, margin=0.99in]{geometry}
\usepackage{amsmath}
\usepackage{amsthm}
\usepackage{bbm}
\usepackage{amssymb}
\usepackage{graphicx}
\usepackage{amsfonts}
\usepackage{xcolor}
\usepackage{pdfpages}
\usepackage{enumerate}
\usepackage{subcaption}
\usepackage[hidelinks]{hyperref}
\usepackage[justification=centering]{caption}
\usepackage{pdfpages}
\usepackage{tikz}
\usetikzlibrary{angles,quotes,calc,intersections}

\usepackage{thmtools}

\begin{document}

\newtheorem{theorem}{Theorem}
\numberwithin{theorem}{section}
\theoremstyle{definition}
\newtheorem{definition}[theorem]{Definition}
\newtheorem{proposition}[theorem]{Proposition}
\newtheorem{example}[theorem]{Example}
\newtheorem{lemma}[theorem]{Lemma}
\newtheorem{corollary}[theorem]{Corollary}
\newtheorem{remark}[theorem]{Remark}
\newtheorem{conjecture}[theorem]{Conjecture}
\newcommand{\uw}{\mathcal{U}(W,X)}
\newcommand{\W}{$(W,S)$}
\newcommand{\ix}{\textbf}
\newcommand{\tr}{\textcolor{red}}
\newcommand{\sg}{$\Sigma$}
\newcommand{\x}{\mathcal{X}}
\newcommand{\C}{\mathcal{C}}
\newcommand{\tw}{T_{w^{-1}}^{-1}}
\newcommand{\mch}{\mathcal{H}}
\newcommand{\aw}{\tilde{A}_2}
\newcommand{\Sh}[1]{\textnormal{Sh}(#1)}
\newcommand{\nSh}[1]{\textnormal{Sh}^\textnormal{c}(#1)}
\newcommand{\PSh}[1]{\textnormal{PSh}(#1)}
\newcommand{\nPSh}[1]{\textnormal{Annex}(#1)}
\newcommand{\wk}{w_k}
\newcommand{\wj}{w_j}
\newcommand{\wi}{w_i}
\newcommand{\B}[1]{B(#1)}
\newcommand{\Hid}{H^{\mathbf{1}}}
\newcommand{\Hinf}{H^\infty}
\newcommand{\cinf}{\mathcal{C}_\infty}

\title{Annexes in affine Coxeter complexes}
\author{Megan Masters}
\maketitle
\begin{abstract}

We introduce the annex of an element $x$ in a Coxeter group as the set of elements $y$ such that $x \nleq y$ with respect to Bruhat order. This notion provides a complementary perspective to the study of Bruhat intervals and their interpretation via folded galleries. We establish general properties of annexes and show that in affine Coxeter groups the annex of any fixed element is finite. In rank-two affine Coxeter complexes, we further describe the geometric structure of annex boundaries using descent sets and configurations of parallel reflections. These results offer a new geometric viewpoint on the structure of the Bruhat order.

\end{abstract}

\section{Introduction}

The theory of buildings, introduced by Jacques Tits in the 1950s \cite{TITS1}, provides a geometric framework for understanding algebraic and Lie-theoretic structures. Originally developed to give descriptions of semisimple algebraic groups and Lie groups \cite{TITS}, buildings have since evolved into a central object of study in geometric group theory, representation theory, and combinatorics \cite{ABRAMENKOBROWN,RON}. Among the most fundamental examples of buildings are Coxeter complexes. These combinatorial and geometric structures encode the reflection symmetries of Coxeter groups and serve as the local objects from which more complicated buildings are assembled.

Given a Coxeter system $(W,S)$, the associated Coxeter complex can be realised as a simplicial complex whose maximal simplices, called alcoves, are in bijection with the elements of $W$ \cite{HUMPHREYS2}. Adjacency of alcoves corresponds to multiplication by simple reflections. In the affine case, this complex admits a geometric embedding in Euclidean space in which reflections in the group correspond to reflections across affine hyperplanes \cite{DAVIS}. This paper focuses on affine Coxeter complexes, particularly those of rank 2.

Within our Coxeter complex, we explore walks, which we call galleries, through the alcoves \cite{RAM}. In this paper, we focus on alcove-to-alcove galleries. Other papers, including Schwer \cite{WILD}, discuss galleries that start and end at vertices. Galleries, and their positive foldings, have applications to a wide range of mathematics. Positive foldings of galleries were first introduced in 2005 by Gaussent and Littelmann \cite{LSGAL}. They can be used to compute Hall--Littlewood polynomials \cite{HL}, and have been used to study Mirkovic--Vilonen polytopes \cite{MVPOLY}. There is also a link between folded galleries and affine Deligne--Lusztig varieties \cite{DEL}. 

Within a Coxeter group, we can define a partial ordering called the Bruhat order. This order was first introduced as an order on Schubert varieties of flag manifolds \cite{EHRESMANN}, and was then extended to general semisimple algebraic groups \cite{CHEVALLEY}. The study of Bruhat order in Weyl groups began in the 1960s \cite{VERMA}. Given a Coxeter group $W$, and elements $x,y\in W$, we say that $x \leq y$ in Bruhat order if, whenever $y$ is written as a reduced expression
\[
y = s_{i_1}s_{i_2}\cdots s_{i_k},
\]
there exists a subword $s_{i_{j_1}} s_{i_{j_2}} \cdots s_{i_{j_m}}$, with $1 \le j_1 < j_2 < \cdots < j_m \le k$, that is a reduced expression for $x$ \cite{COMB}. 

Associated to an element $w\in W$ is its \ix{shadow}, defined as the set of all group elements that can arise as the endpoint of a positively folded gallery of type 
$w$. In the case of the trivial orientation (which is the setting of this paper), the shadow of $w$ coincides precisely with the Bruhat interval $[1,w]$. Thus the geometry of folded galleries provides a visual and combinatorial model for Bruhat order. See \cite{WILD} for a good overview of the applications of shadows.

While the lower Bruhat interval $[1,w]$ admits this simple description via shadows \cite{SHA}, more general Bruhat intervals $[y,z]$ are less directly accessible from the gallery perspective. Understanding such intervals is a classical problem with broad implications. For example, Bruhat intervals play a fundamental role in the geometry of Schubert varieties \cite{SCHUBERT} and in the theory of Kazhdan–Lusztig polynomials \cite{KL}.

This paper investigates the geometric properties of the upper Bruhat interval $[x,\infty)$. Here, we are asking for which $y$ does $x$ lie in the shadow of $y$? Equivalently, in terms of the Bruhat order, for which $y$ is $x \leq y$? This leads naturally to the notion of the \ix{preshadow} of $x$, the set of all elements whose shadows contain $x$. In infinite Coxeter groups, and in particular in affine Coxeter groups, preshadows are infinite sets. Consequently, it is often easier to study their complements. This motivates the central definition of the paper. 

\begin{definition}
    Given an element $x \in W$, we define the \ix{annex of $x$}, $\nPSh{x}$, to be the set of all alcoves $y$ such that $x$ does not lie in the shadow of $y$.
\end{definition}

In the case that our Coxeter group is affine, the following result holds. 

\begin{theorem}\label{finite}
    Consider an affine Coxeter group $W$. Let $w\in W$. Then the annex of $w$ is a finite set.
\end{theorem}

We then explore the geometry of this object, restricting to rank 2 Coxeter complexes. These form tilings of the plane. A fundamental component of this exploration is the following theorem, describing the effect of adjacent parallel reflections on the length of an element.  

\begin{theorem}\label{pm1}
    Suppose we have three adjacent parallel reflections $H_{r_0}$, $H_{r_1}$ and $H_{r_2}$, such that 
    \[\ell(r_2r_1x)=\ell(r_1x)+1=\ell(x)+2.\] Then $\ell(r_0x)=\ell(x)\pm 1$. 
\end{theorem}

This theorem allows us to form an inductive argument on the boundary of the annex, leading to the following theorem.

\begin{theorem} \label{theorem}
        Let $W$ be a rank 2 affine Coxeter complex. Let $w\in W$ and $\ell(ws_i)<\ell(w)$. Let $r_1,\hdots,r_n\in W$ be a sequence of adjacent parallel reflections such that \[\ell(r_t\cdots r_1ws_i)=\ell(r_{t-1}\cdots r_1 ws_i)+1,\] 
    and $r_{t}\cdots r_1ws_i\neq r_{t-1}\cdots r_1w$, for all $1\leq t\leq n$. Suppose there are no hyperplanes parallel to $H_{r_1}$ between $w$ and $H_{r_1}$. Then $r_n\cdots r_1ws_i\in \nPSh{w}$. 
\end{theorem}

We also prove that $\ell(r_{t-1}\cdots r_1 w)<\ell(r_t\cdots r_1w)$ for each $t$. It is then easy to conclude that $r_n\cdots r_1ws_i$ must lie on the boundary of the annex. 

In this paper, in Section \ref{1}, we give the background of Coxeter complexes and galleries, defining folds and trivial shadows. We see how Coxeter complexes are linked to root systems, and how shadows are related to the Bruhat order on the group. Then, in Section \ref{2}, we define the new notation of annexes. 
We show some results of annexes, including that, within affine complexes, these sets are finite. Lastly, in Section \ref{3}, we prove the theorem that completes our description of the boundaries of the annexes in rank 2.

\section{Background}\label{1}
We start with a Coxeter group $(W,S)$ with presentation
\[W=\langle s_1,\hdots , s_n \mid (s_is_j)^{m_{ij}}=1\rangle,\]
where $m_{ii}=1$ for all $i$ and $m_{ij}\in \mathbb{Z}_{\geq 2}\cup\{\infty\}$ for all $i\neq j$. Let $I$ be an indexing set for $S$. The \ix{rank} of the Coxeter group with presentation above is $n$, the number of generators in the presentation. We can construct the Coxeter complex of this group. The largest-dimensional simplices are called \ix{alcoves}. They have dimension $n-1$, and they are in bijection with elements of $W$. Then any two alcoves $x,y$ are adjacent if and only if $x=ys_i$, for some $s_i$ in the generators. We call the simplices between alcoves \ix{panels}. These have dimension $n-2$. If a panel lies in $x$ and $y=xs_i$, we say that the panel has \ix{type} $i$.   See \cite {RON} for a detailed description of the construction.

\subsection{Affine Coxeter groups}

We now fix the following notation for affine Coxeter groups. We will see later that restricting to affine Coxeter groups allows us to prove nice results on the annexes. In particular, these sets are finite. 

\begin{definition}
 An \ix{affine Coxeter group} is an infinite Coxeter group that has a normal abelian subgroup with finite quotient. 
\end{definition}

\begin{example}
    Consider the example of the affine Coxeter group $\tilde{A}_2$. This has presentation
\[\tilde{A}_2=\langle s_0,s_1,s_2\mid s_i^2=1, (s_0s_1)^3=(s_0s_2)^3=(s_1s_2)^3=1\rangle.\]
The Coxeter complex associated with this Coxeter system is the tiling of the plane by equilateral triangles. 
\end{example}
\begin{example}
    Affine Coxeter groups often arise by looking at the semidirect product of a finite Coxeter group and a translation group. 
\end{example}

The next proposition follows from the classification of Dynkin diagrams for finite and affine Coxeter groups. See \cite[Chapter 4]{HUMPHREYS2} for details. 

\begin{proposition}\label{subgroupsfinite}
    Let $W$ be an affine Coxeter group generated by some set of reflections $S$. Let $S'\subsetneq S$ be any proper subset of $S$, and let $W_{S'}$ be the subgroup of $W$ generated by $S'$. Then $W_{S'}$ is finite. 
\end{proposition}

\subsection{Reflections and walls}\label{hyperplanes}
We wish to examine how the structure of the underlying Coxeter group appears in the Coxeter complex. Here we will formalise the concept of walls in the Coxeter complex, and we will fix the set of elements of our Coxeter group which are in bijection to reflections in the geometric realisation of the Coxeter complex. 

\begin{definition}\label{walls}
    An element $r\in W$ is called a \ix{reflection} if it is a conjugate of an element $s_i\in S$. The \ix{wall} $H_r$ of a reflection $r$ is the set of simplices in the Coxeter complex that are fixed by $r$ when $r$ acts on the complex by left multiplication. Then $H_r$ is a subcomplex of codimension 1.
\end{definition}

By this definition, a panel lies in a wall $H_r$ if and only if it is fixed by the reflection $r$. 

\begin{example}
    In the case that the Coxeter complex is affine, the walls are in bijection to hyperplanes of the geometric realisation. 
\end{example}

\begin{proposition}\textnormal{\cite[Proposition 2.6]{RON}}
    There is a bijection between the set of reflections of a Coxeter group, and the set of walls in the corresponding Coxeter complex.
\end{proposition}

\begin{proof}
    Let $p$ be an $i$-panel of a chamber $x$. Then the unique reflection that fixes $p$ and interchanges $x$ and $xs_i$ is the map $r=xs_ix^{-1}$. So the map from reflections to walls is a bijection.
\end{proof}

\begin{definition}
    Let $H$ be a hyperplane of a Coxeter complex. This hyperplane divides the set of alcoves into two sets. Define the \ix{identity half} $\Hid$ to be the set containing the identity alcove, and the \ix{infinity half} $\Hinf$ to be the set not containing the identity alcove.
\end{definition}

\begin{remark}
    Note that, if $H=H_r$, 
    \[\Hid=\{x\in W\mid \ell(rx)>\ell(x)\} \textnormal{ and } \Hinf=\{x\in W\mid \ell(rx)<\ell(x)\}.\]
\end{remark}

\subsection{Root systems}\label{roots}

Next, we recall the fact that affine Coxeter complexes arise naturally from root systems. 
We fix the following notation, following the conventions of Humphreys \cite[Chapter 3]{HUMPHREYS}.  

\begin{definition}
    Consider an $n$-dimensional vector space $E$ over $\mathbb{R}$, with inner product $\langle\cdot,\cdot\rangle$. For $\alpha\in E\backslash\{0\}$, define
    \[\alpha^\vee=\dfrac{2\alpha}{\langle\alpha,\alpha\rangle}.\]
    Also, define the set
    \[H_\alpha = \{x\in E\mid \langle x,\alpha\rangle=0\}.\]
    The \ix{orthogonal reflection} in $H_\alpha$ is the map $s_\alpha :E\to E$ with
    \[s_\alpha(x)=x-\langle x,\alpha\rangle\alpha^\vee.\]
\end{definition}

Now let us fix a root system $\Phi$ of $E$. 
We wish to consider both the finite and affine Weyl groups connected to this root system. 
\begin{definition}
    The \ix{(finite) Weyl group} $W_0$ of a root system $\Phi$ is the group generated by the reflections $s_\alpha$ for $\alpha\in R$. 
\end{definition}

The hyperplanes $H_\alpha$ for $\alpha\in R$, which are the hyperplanes orthogonal to the roots, partition $E$ into finitely many open regions, called \ix{Weyl chambers}.

We now extend the definition of our reflections to create the corresponding infinite group. For each $k\in\mathbb{Z}$, we fix the reflection $s_{\alpha;k}:E\to E$ as
\[s_{\alpha;k}(x)=x-(\langle x,\alpha\rangle-k)\alpha^\vee.\]
Now for each root $\gamma$ in our root system, and for each $k\in\mathbb{Z}$, we have the hyperplane
\[H_{\gamma,k}=\{x\in E\mid \langle x, \gamma\rangle = k\}.\]

\begin{definition}
    The \ix{affine Weyl group} $W$ of a root system $\Phi$ is the group generated by the reflections $s_{\alpha;k}$ for $\alpha\in R$, $k\in \mathbb{Z}$. 
\end{definition}

\begin{remark} All affine Coxeter complexes can be obtained from this root system construction. The hyperplanes correspond to the walls of the complex, as described above. See \cite[Chapter 4]{HUMPHREYS2} for details. 
\end{remark}

\subsection{Galleries}

Now we will take walks around our Coxeter complex, moving from alcove to alcove. We will fix how we can fold these walks, and recall the links to Bruhat order in the group. In this section, we largely follow \cite[Chapter 4]{SHA}.

\begin{definition}\label{comb.gallery}
    Given a Coxeter complex \sg, a \ix{combinatorial gallery} is a sequence
    \[\gamma = (c_0,p_1,c_1,p_2,\hdots  ,p_n,c_n),\]
    where the $c_i$ are alcoves (or chambers if our Coxeter complex if finite) and the $p_i$ are panels of \sg, such that $p_i$ is contained in $c_{i-1}$ and $c_{i}$ for all $i=1,\hdots ,n$. The length of the combinatorial gallery $\gamma$ is $n$. This counts how many alcoves there are in the sequence. Then $\gamma$ is \ix{minimal} if there does not exist a shorter gallery starting at $c_0$ and ending at $c_n$. 
\end{definition}

\begin{definition}\label{gallerytype}
    Let $\gamma=(c_0,p_1,c_1,p_2,\hdots  ,p_n,c_n)$ be a combinatorial gallery. The \ix{type} of $\gamma$ is the word $f=j_1\hdots j_n$ in $I$, where panel $p_i$ has type $j_i$. 
\end{definition}

The following proposition classifies the walls a minimal gallery crosses. See \cite[Chapter 2]{RON} for details. 

\begin{proposition}\label{gallerycrossing}
    Let $W$ be a Coxeter group. Let $H_r$ be a wall, $x\in W$ an alcove, and $\gamma$ a minimal gallery from 1 to $x$. If $\ell(rx)<\ell(x)$, $\gamma$ crosses $H_r$ exactly once. If $\ell(rx)>\ell(x)$, $\gamma$ does not cross $H_r$. 
\end{proposition}

\subsubsection{Folded galleries}\label{folded}
Within our definition of a gallery, it is possible that two consecutive alcoves are actually the same alcove. We call galleries with this property folded galleries. To classify these galleries, we will introduce the $(W,S)$-type of a gallery. This encodes the types of all the panels we crossed over or touched in our gallery. We will also define the decorated $(W,S)$-type. This also encodes the location of folds in our gallery. The decorated $(W,S)$-type, along with the footprint function and starting alcove, can be used to calculate the end alcove of our gallery.

\begin{definition}
    Given a combinatorial gallery $\gamma=(c_0,p_1,c_1,p_2,\hdots  ,p_n,c_n)$ of \sg, as in Definition \ref{comb.gallery}, we say that $\gamma$ is \ix{folded} (or \ix{stammering}) if, within $\gamma$, we can find an index $i$ such that $c_i=c_{i-1}$. Then we say that $\gamma$ has a \ix{fold} at panel $p_i$. If there are no folds in $\gamma$, we say that $\gamma$ is \ix{unfolded} (or \ix{non-stammering}).  
\end{definition}

\begin{definition}
    Given a gallery $\gamma=(c_0,p_1,c_1,p_2,\hdots  ,p_n,c_n)$ of length $n+1$, define the \ix{fold set} F$(\gamma)$ to be the subset of $\{1,\hdots ,n\}$ such that $i\in$ F$(\gamma)$ if and only if $\gamma$ has a fold at panel $p_i$. 
\end{definition}

To represent a gallery, we draw a path that passes through every chamber and panel in the gallery in the geometric representation of our Coxeter complex. We draw an arrow towards the end alcove of our gallery. 

\begin{figure}[!htbp]
    \begin{center}
    \includegraphics[scale=0.2]{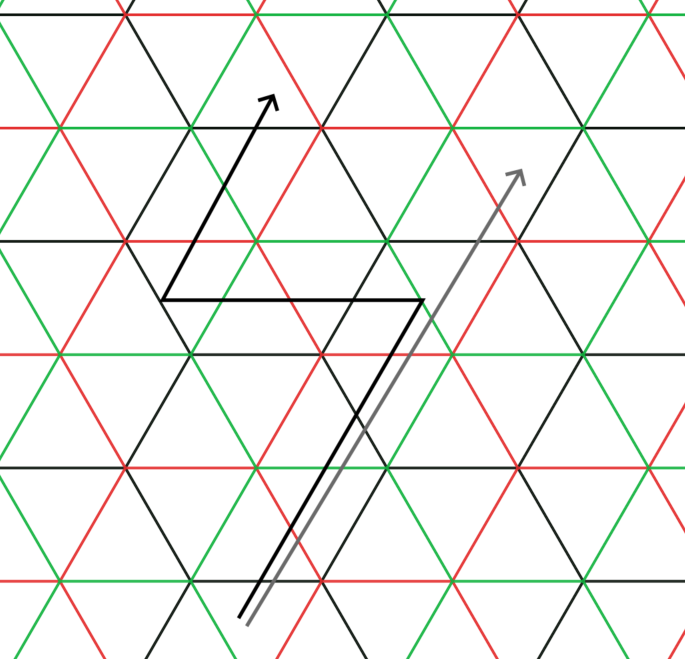}
    \end{center}
    \caption{Two galleries in $\tilde{A}_2$. The grey gallery is unfolded, and the black gallery is folded at two panels.}
\end{figure}

We note that, as $W$ has a natural left action on \sg, $W$ also acts on the set of galleries in \sg. For instance, $x\in W$ sends $\gamma = (c_0,p_1,c_1,\hdots ,p_n,c_n)$ to the gallery $x\cdot \gamma = (xc_0,xp_1,xc_1,\hdots ,xp_n,xc_n)$. 

\subsubsection{Galleries and words}
We now want to map galleries in our Coxeter complex to words in our generating set. We do this by recalling the definition of the (decorated) $(W,S)$-type of a gallery. 

\begin{definition}
    Consider a gallery $\gamma = (c_0,p_1,c_1,\hdots ,p_n,c_n)$. Let panel $p_i$ of $\gamma$ have type ${j_i}\in I$. We define its \ix{$(W,S)$-type} $\tau(\gamma)$ as the word 
    \[\tau(\gamma):=s_{j_1}\hdots s_{j_n}.\]
 
\end{definition}

Note that, from Definition \ref{gallerytype}, if $\gamma$ has type $f$, then $\gamma$ has $(W,S)$-type $s_f$. 

\begin{definition}
    The \ix{decorated $(W,S)$-type} $\tilde\tau(\gamma)$ of a gallery $\gamma = (c_0,p_1,c_1,\hdots ,p_n,c_n)$ is the decorated word
    \[\tilde\tau(\gamma):= s_{j_1}\hdots \tilde{s}_{j_i}\hdots s_{j_n},\]
    where we place a tilde on the elements $s_{j_i}$ of the word that correspond to a fold $c_{i-1}=c_i$ of the gallery. 
    
\end{definition}
We now state some simple lemmas that start to show the link between folded galleries and Bruhat order. 

\begin{lemma}\cite[Lemma 4.8]{SHA}
    Let $c_0$ be a fixed alcove in our Coxeter complex \sg. 
    \begin{enumerate}[(i)]
        \item There is a bijection between words in $S$ and unfolded galleries starting at $c_0$.
        \item There is a bijection between decorated words in $S$ and galleries starting at $c_0$. 
    \end{enumerate}
\end{lemma}

\begin{lemma}\cite[Lemma 4.9]{SHA}
    Let $\gamma$ be a gallery. Then
    \begin{enumerate}
        \item $F(\gamma)=\varnothing$ if and only if $\tau(\gamma)=\tilde{\tau}(\gamma).$
        \item $\gamma$ is minimal if and only if $F(\gamma)=\varnothing$ and $\tau(\gamma)$ is reduced.
    \end{enumerate}
\end{lemma}

\begin{definition}
    Consider a gallery $\gamma = (c_0,p_1,c_1,\hdots ,p_n,c_n)$ in \sg. We create a new gallery, called the \ix{footprint} ft$(\gamma)$ of $\gamma$, by deleting all pairs $(p_i,c_i)$ such that the letter $s_{j_i}$ has a tilde in $\tilde{\tau}(\gamma)$. 
\end{definition}

The following lemma follows from the definition of type of a gallery. 

\begin{lemma}\label{type}
    Let $\gamma=(c_0,p_1,c_1,p_2,\hdots ,p_n,c_n)$ be an unfolded combinatorial gallery with type $w=s_{i_1}\hdots s_{i_n}$. Then $c_n=c_0\cdot w$. Furthermore, if there is another unfolded gallery between $c_0$ and $c_n$ with type $v$, then $w=v$ as group elements. 
\end{lemma}

\subsubsection{Folding and unfolding galleries}\label{foldingsandunfolding}
Now, we want to take a folded gallery, and form an unfolded gallery of the same $(W,S)$-type. Similarly, if we have an unfolded gallery, we want to form a collection of folded galleries all of the same $(W,S)$-type. The underlying idea is that we reflect sections of our gallery over hyperplanes in the Coxeter complex. Here we will formalise this. 

Figure \ref{folds} provides an example. If we wish to fold or unfold at a panel, we must fix the gallery up to that panel, and then reflect the rest of the gallery along the hyperplane containing that panel.

\begin{figure}[!htbp]
    \begin{center}
    \includegraphics[scale=0.25]{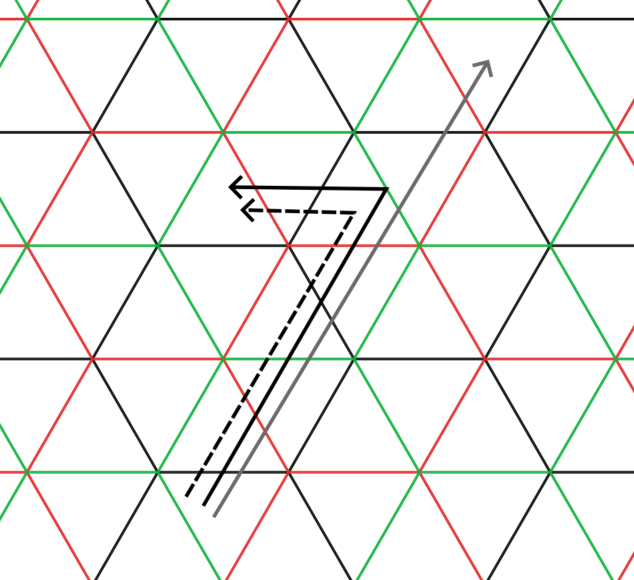}\\
    \end{center}
    \caption{Here we have a gallery shown in black. The grey line shows the corresponding unfolded gallery and the dashed line shows its footprint. }
    \label{folds}
\end{figure}

\begin{definition}
    Consider a gallery $\gamma = (c_0,p_1,c_1,\hdots ,p_n,c_n)$. Let $H_i$ be the wall containing the panel $p_i$, and let $r_i$ be the reflection across $H_i$. For $i=1,\hdots ,n$, let
    \[\gamma^i:=(c_0,p_1,\hdots ,p_i,r_ic_i,r_ip_{i+1},r_ic_{i+1},\hdots ,r_ip_n,r_ic_n).\]
    If $\gamma$ was folded at panel $p_i$, we call $\gamma^i$ an \ix{unfolding of }$\gamma$ at $p_i$. Otherwise, we call it a \ix{folding}.
\end{definition}

\begin{lemma}\cite[Lemma 4.16]{SHA}
    For all $i=1,\hdots ,n$, $\tau(\gamma)=\tau(\gamma^i)$. So folding and unfolding does not change the gallery $(W,S)$-type. Also, $(\gamma^i)^i=\gamma$.
\end{lemma}

\begin{lemma}\cite[Lemma 4.17]{SHA}
    For all $i,j\in\{1,\hdots,n\}$, $(\gamma^i)^j=(\gamma^j)^i$, i.e.\ foldings commute.
\end{lemma}

Because of this property, we are able to define a \ix{multifolding} with respect to a subset $J$ of $\{1,\hdots ,n\}$ as the (un-)foldings $\gamma^J$. Now multifolding does not affect the $(W,S)$-type. Then the set of folds of $\gamma^J$ will be the symmetric difference between $F(\gamma)$ and $J$. So, if $J$ and $K$ are subsets of $\{1,\hdots ,n\}$, $(\gamma^J)^K=\gamma^{J\Delta K}$. The following corollary now follows.

\begin{corollary}
    Given any gallery $\gamma$, there is a subset $J\subset \{1,\hdots ,n\}$ such that $\gamma^J$ is unfolded, and $\gamma$ and $\gamma^J$ have the same $(W,S)$-type.
\end{corollary}

Now we fix a bijection between $W$ and the alcoves of our Coxeter complex. We denote the alcove associated to the identity element by 1.

\begin{definition}
    Let $w,x\in W$. Let $\gamma$ be any gallery starting in 1 and ending in $w$. We write $w\rightharpoonup x$ if there is a folding $\gamma^I$ of $\gamma$ such that the the end alcove of $\gamma^I$ is $x$.
\end{definition}

\begin{remark}
    This definition does not depend on the choice of gallery. This is because the trivial orientation, where all folds are allowed, is braid invariant. See \cite[Chapter 5]{SHA} for details. 
\end{remark}

\begin{definition}
    Consider a Coxeter system $(W,S)$. Let $w\in W$. The \ix{(trivial) shadow} of $w$ is the set 
    \[\textnormal{Sh}(w):=\{x\in W\mid w\rightharpoonup x\}.\] 
\end{definition}

\begin{remark}
    In this paper, we are restricting our discussion to the trivial shadow. We can construct more complicated shadows by placing an orientation on our Coxeter complex. This will restrict our choice of foldings. See \cite{SHA} for more details. 
\end{remark}

\begin{definition}
    Let $x=s_1\hdots s_n$ be a reduced expression for $x\in W$. Let $y\in W$. We say that $y\leq x$ if there exists a reduced expression for $y$ of the form $s_{i_1}\hdots s_{i_k}$ with $1\leq i_1\leq\hdots \leq i_j\leq n$. This ordering is called the \ix{Bruhat order}.
\end{definition}

The following proposition recalls the link between Bruhat order and the shadow of an element.

\begin{proposition}\cite{SHA}
   For $x,y\in W$, $y\leq x$ if and only if $y\in \Sh{x}$.
\end{proposition}

\begin{proof}
    Assume $y\leq x$. Let $x=s_1\hdots s_n$ and $y=s_{i_1}\hdots s_{i_k}$. Let $\gamma$ be the gallery of type $s_1\hdots s_n$ starting at 1. Then $\gamma$ ends in $x$. Let $I=\{1,\hdots,n\}-\{i_1,\hdots,i_k\}$. Then $\gamma^I$ ends in $y$. 

    Now assume that $y\in \Sh{x}$. Let $\gamma$ be a gallery of type $s_1\hdots s_n$ starting at 1 and ending in $x$. Then, for some $I$, the footprint of $\gamma^I$ ends in $y$. Define $i_1,\hdots,i_k$ such that $I=\{1,\hdots,n\}-\{i_1,\hdots,i_k\}$ and $1\leq i_1\leq\hdots \leq i_j\leq n$. Then $y=s_{i_1}\hdots s_{i_k}$, and so $y\leq x$. 
\end{proof}

Therefore, we have that the trivial Bruhat interval $[1,x]$ is equal to the shadow of $x$.

\section{The annex of an alcove}\label{2}
We have seen that the trivial Bruhat interval has a nice description in terms of foldings of galleries, and therefore shadows. We would like to have a similar description for any Bruhat interval $[y,z]$. These intervals arise in many different areas of mathematics. For example, they play an important role in the exploration of Schubert varieties \cite{SCHUBERT}, and are fundamental to Kazhdan--Lusztig polynomials \cite{KL}. See \cite{COMB} for details of the applications of Bruhat order. If $x\in [y,z]$, we have that $x\in [1,z]$, and $y\in [1,x]$. The first statement, that $x\in [1,z]$, is the trivial shadow. But the second statement, that $y\in [1,x]$, is asking us to find all $x$ such that $y$ is in the trivial shadow of $x$. We explore the geometry of this set in this paper.

\begin{definition}
    Let $x$ be an alcove in a Coxeter complex \sg. Define the \ix{preshadow} of $x$ to be
    \[\textnormal{PSh}(x)=\{y\in\Sigma\mid x\in \textnormal{Sh}(y)\}.\]
\end{definition}

\begin{figure}[htbp!]
    \centering
    \includegraphics[width=0.3\linewidth]{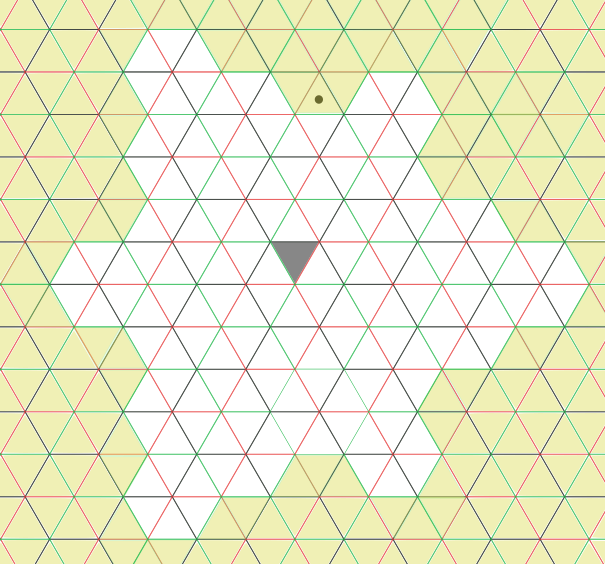}
    \caption{The alcoves shaded yellow are in the preshadow of $x=s_0s_1s_2s_0s_1s_0s_2$.}
\end{figure}

In the case of infinite Coxeter groups, these sets are infinite. So we instead want to classify every alcove that is not in the preshadow. We call this set the annex, and we will see in Theorem \ref{finite} below that this is a finite set. 
\begin{definition}
    Define the \ix{annex} of $x$, $\nPSh{x}$, to be the complement of the preshadow of $x$, i.e.\
    \[\nPSh{x}=\{y\in\Sigma\mid x\notin \textnormal{Sh}(y)\}.\]
\end{definition}
\begin{figure}[htbp!]
    \centering
    \includegraphics[width=0.3\linewidth]{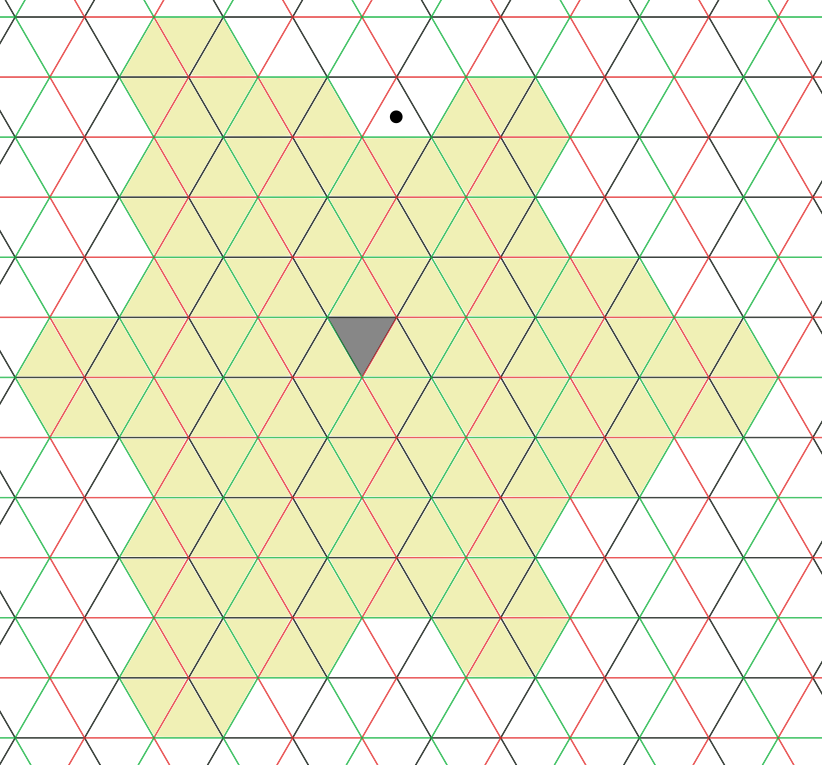}
    \caption{The annex of $x=s_0s_1s_2s_0s_1s_0s_2$ is the set of alcoves shaded yellow.}
    \label{alcovepic}
\end{figure}

\begin{remark}
    Note that $y\in \nPSh{x}$ if and only if $x\nleq y$ as group elements, with respect to the Bruhat order.
\end{remark}

\begin{remark}
    We want to calculate the Bruhat interval $[x,y]$. Using our notation, we have
    \[[x,y]=\Sh{y}\cap\PSh{x}=\Sh{y}-\nPSh{x}. \]
\end{remark}
\subsection{Fundamental properties of the annex}
We start by proving some fundamental properties of the annex for all Coxeter groups. 

\begin{lemma}\label{xinw}
    If $x\in \PSh{w}$, then $\PSh{x}\subseteq \PSh{w}$, and hence $\nPSh{w}\subseteq \nPSh{x}$.
\end{lemma}
\begin{proof}
    Let $y\in \PSh{x}$. Then $x\leq y$. Now $w\leq x$ as $x\in \PSh{w}$, so $w\leq y$, as the Bruhat order is a partial order. Thus, $y\in\PSh{w}$. 
\end{proof}

\begin{proposition}\label{rxinx}
    Let $W$ be a Coxeter group, and let $r\in W$ be a reflection. Let $x\in W$ such that $\ell(rx)>\ell(x)$. Then $rx\in \PSh{x}$. 
\end{proposition}
\begin{proof}
    Let $H$ be the wall fixed by $r$. Consider any minimal gallery $\gamma=(1,p_1,c_1,\hdots,p_n,c_n)$ from the identity to $c_n=rx$. As $r$ reduces the length of $rx$, $\gamma$ must pass over $H$ exactly once by Proposition \ref{gallerycrossing}, say at panel $p_i$. Fold the gallery $\gamma$ at $p_i$. This gives the gallery 
    \[(1,p_1,c_1,\hdots,p_i,rc_i,rp_{i+1},rc_{i+1},\hdots,rp_n,rc_n)\]
    that starts at the identity and ends at $r(rx)=x$. Therefore, $rx\in \PSh{x}$.
\end{proof}

\begin{corollary}\label{rxinxcor}
     Let $W$ be a Coxeter group, and let $r\in W$ be a reflection. Let $x,w\in W$ be such that $x\in \PSh{w}$ and $\ell(rx)>\ell(x)$. Then $rx\in \PSh{w}$.
\end{corollary}
\begin{proof}
    By Proposition \ref{rxinx}, $rx\in\PSh{x}$. Now, as $x\in\PSh{w}$, by Lemma \ref{xinw}, $\PSh{x}\subseteq \PSh{w}$. So $rx\in \PSh{w}$. 
\end{proof}

\begin{corollary}
     Let $W$ be a Coxeter group, and let $r\in W$ be a reflection. Let $x,w\in W$ such that $x\in \nPSh{w}$ and $\ell(rx)<\ell(x)$. Then $rx\in \nPSh{w}$.
\end{corollary}

We state the following lemma, which will be useful in forming an inductive argument to classify the annex . 

\begin{lemma}\textnormal{\cite[Lemma 2.1]{BRUHATCOSETS}}\label{Bruhatcosets}
   Let $(W,S)$ be a Coxeter group, $s\in S$ and $u,w\in W$ such that $w\leq u$. If $\ell(w)<\ell(ws)$ and $\ell(us)<\ell(u)$, then $ws\leq u$ and $w\leq us$.  
\end{lemma}
We can now classify the types of panels that can appear in the boundary of $\nPSh{w}$. 

\begin{proposition}\label{boundarypanels}
    Let $(W,S)$ be a Coxeter system. Consider an element $w\in W$. If $\ell(ws_i)>\ell(w)$ for some $s_i\in S$, then the boundary of $\nPSh{w}$ does not contain any panels of type $s_i$.  
\end{proposition}
\begin{proof}
    Let us assume that $\ell(ws_i)>\ell(w)$ and the boundary of $\nPSh{w}$ contains a panel of type $s_i$. So there exists $y\in W$ such that $y\in \nPSh{w}$ but $ys_i\in \PSh{w}$. So $w\leq ys_i$, but $w\nleq y$. If $\ell(y)>\ell(ys_i)$, by the exchange property of Coxeter groups, $ys_i\leq y$, so $w\leq y$, which is a contradiction. So $\ell(y)<\ell(ys_i)$, and hence $\ell((ys_i)s_i)<\ell(ys_i)$. By assumption $\ell(w)<\ell(ws_i)$, so, applying Lemma \ref{Bruhatcosets}, we have that $w\leq ys_is_i=y$. This is again a contradiction to $y\in \nPSh{w}$. Therefore, $\nPSh{w}$ contains no panels of type $s_i$.
\end{proof}

\begin{corollary}\label{corollary}
    If $y\in\nPSh{w}$ and $\ell(ws_i)>\ell(w)$, then $ys_i\in\nPSh{w}$.
\end{corollary}

\begin{proof}
    If $ys_i\notin \nPSh{w}$, then as $y\in \nPSh{w}$, the boundary of $\nPSh{w}$ contains a panel of type $s_i$. But $\ell(ws_i)>\ell(w)$, which contradicts Proposition \ref{boundarypanels}. 
\end{proof}

So we have now classified the possible panels on the boundary of the annex. Next, we introduce the concept of the right descent set, and prove that the annex is finite. 
\begin{definition}
    Let $w\in W$. Define the set 
    \[D_R(w)=\{i\in I\mid \ell(ws_i)<\ell(w)\}\]
    to be the \ix{right descent set} of $w$. 
\end{definition}

\begin{proposition}\label{inclusion}
    Let $w\in W-\{1\}$. Let $k\notin D_R(w)$. Let $W_{S-\{k\}}=\langle s_i\mid i\in I-\{k\}\rangle$. Then \[\nPSh{ws_k}\subseteq \nPSh{w}\cdot W_{I-\{k\}}=\{xy\mid x\in \nPSh{w}, y\in W_{I-\{k\}}\}.\]
   
\end{proposition}

\begin{proof}
    Consider $x\in \nPSh{ws_k}$. Now we can assume that $x \notin \nPSh{w}$, otherwise we can immediately conclude that $x\in\nPSh{w}\cdot W_{I-\{k\}}$, by taking $1\in W_{I-\{k\}}$. So $x$ must contain a subword equal to $w$, but cannot contain a subword equal to $ws_k$. 

    Let $x=yy'$, where $y$ and $y'$ are minimal words in the generators, and $y$ is minimal with respect to the property that $w\leq y$. Then $y$ must be nonempty as $w\neq e$, and cannot end in $s_k$ as $\ell(ws_k)>\ell(w)$. So $y=zs_i$ for some $i\in I-\{k\}$ with $\ell(z)<\ell(y)$. As $y$ was minimal with respect to the property that $w\leq y$, we have that $w\nleq z$. So $z\in \nPSh{w}$.

    Now, consider the word $y'$. If $y'$ contained $s_k$ then we can clearly see that $x$ would contain a subword equal to $ws_k$, which is a contradiction. So $y'\in W_{I-\{k\}}$. Hence,
    $x=zs_iy',$
    with $z \in \nPSh{w}$ and $s_iy'\in W_{I-\{k\}}$. Hence, $x\in\nPSh{w}\cdot W_{I-\{k\}}$.
\end{proof}

We are now able to prove Theorem \ref{finite}, which states that if we have an affine Coxeter group, then the annex of any element is finite. 

\begin{proof}(of Theorem \ref{finite})
    By induction on the length of $w$. If $\ell(w)=0,$ then $\nPSh{w}=\nPSh{1}=\varnothing$. If $\ell(w)=1$, then $w=s_k$ for some $k\in I$. Therefore, if a word contains $s_k$, it must be in the preshadow, and so 
    \[\nPSh{w}\subseteq W_{I-\{k\}},\]
    which, as $W$ is affine, is a finite set, by Proposition \ref{subgroupsfinite}. 

    Now let us assume the claim holds for any $w\in W$ with $\ell(w)<t$, with $t>0$. Consider an element $z\in W$, with $\ell(z)=t$. Then $z\neq e$, so choose $k\in D_R(z)$. Let $w=zs_k$, so $\ell(w)=t-1<t$. Now,
    \[\nPSh{z}=\nPSh{ws_k}\subseteq \nPSh{w}\cdot W_{I-\{k\}}.\]
    But $W_{I-\{k\}}$ is finite as $W$ is affine, and $\nPSh{w}$ is finite by induction, so $\nPSh{z}$ is finite. 
\end{proof}
Lastly, we show that in the specific case that $D_R(ws_k)=\{k\}$, this statement is not just an inclusion but an equality. 
\begin{proposition}\label{induction}
    Take $w\in W-\{e\}$. Let $k\notin D_R(w)$. If $D_R(ws_k)=\{k\}$, then $\nPSh{ws_k}=\nPSh{w}\cdot W_{I-\{k\}}$. 
\end{proposition}

\begin{proof}
    By Proposition \ref{inclusion}, $\nPSh{ws_k}\subseteq \nPSh{w}\cdot W_{I-\{k\}}$. So let $x\in\nPSh{w}\cdot W_{I-\{k\}}$. Let $x=yy'$, for $y\in\nPSh{w}$ and $y'\in W_{I-\{k\}}$. Firstly we note that if $x\in\nPSh{w}$, and so $\ell(y')=0$, then $x\in\nPSh{ws_k}$, as $\ell(ws_k)>\ell(w)$. Now we make an induction argument on the length of $y'$. So assume that, if $\ell(y')=n-1$, then $x=yy'\in \nPSh{ws_k}$. 
    
    Now assume that $\ell(y')=n>0$. Then take $i\in {I-\{k\}}$ such that $\ell(y's_i)<\ell(y')$. By assumption, $yy's_i\in \nPSh{ws_k}$. Then, as $J_{ws_k}=I-\{k\}$, $\ell(ws_ks_i)>\ell(ws_k)$, and so $yy's_is_i\in \nPSh{ws_k}$ by Corollary \ref{corollary}. Therefore, $x\in \nPSh{ws_k}$. 
    
    So, by induction, we must have that $x\in \nPSh{ws_k}$ for any $x\in\nPSh{w}\cdot W_{I-\{k\}}$.
\end{proof}

Figure \ref{figgrid1} demonstrates the inductive application of this proposition. 

\begin{figure}[htbp]
     \centering
     \begin{subfigure}[b]{0.49\textwidth}
         \centering
         \includegraphics[width=\textwidth, height=5cm, keepaspectratio, trim={0 0 10 0}, clip]{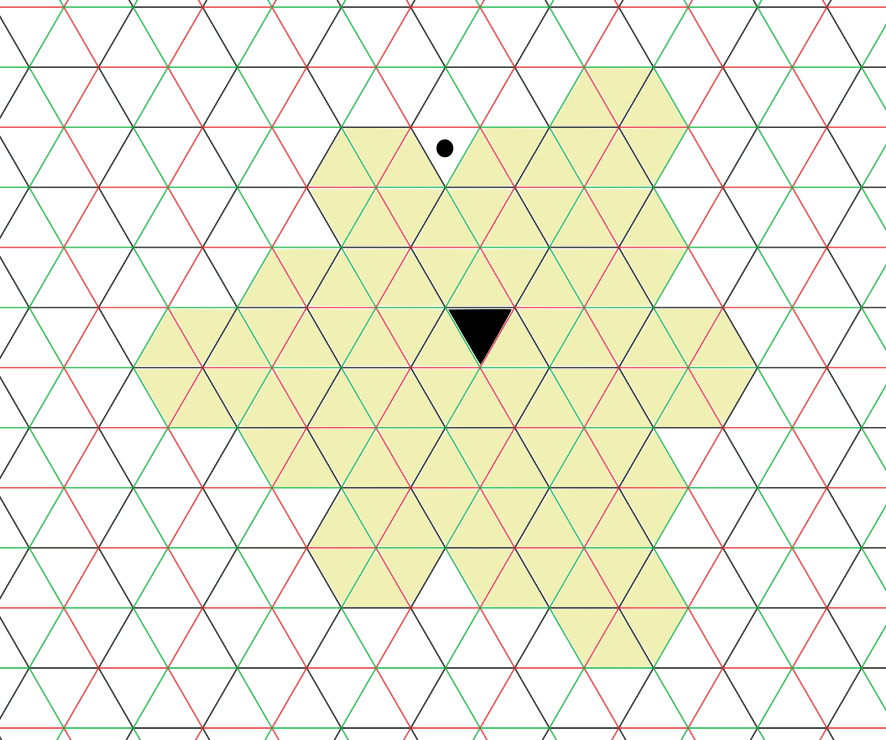}
         \caption{The annex of the alcove $s_0s_2s_1s_0s_2s_0$.}
         
     \end{subfigure}
     \hfill 
     \begin{subfigure}[b]{0.49\textwidth}
         \centering
         \includegraphics[width=\textwidth, height=5cm, keepaspectratio, trim={0 0 10 0}, clip]{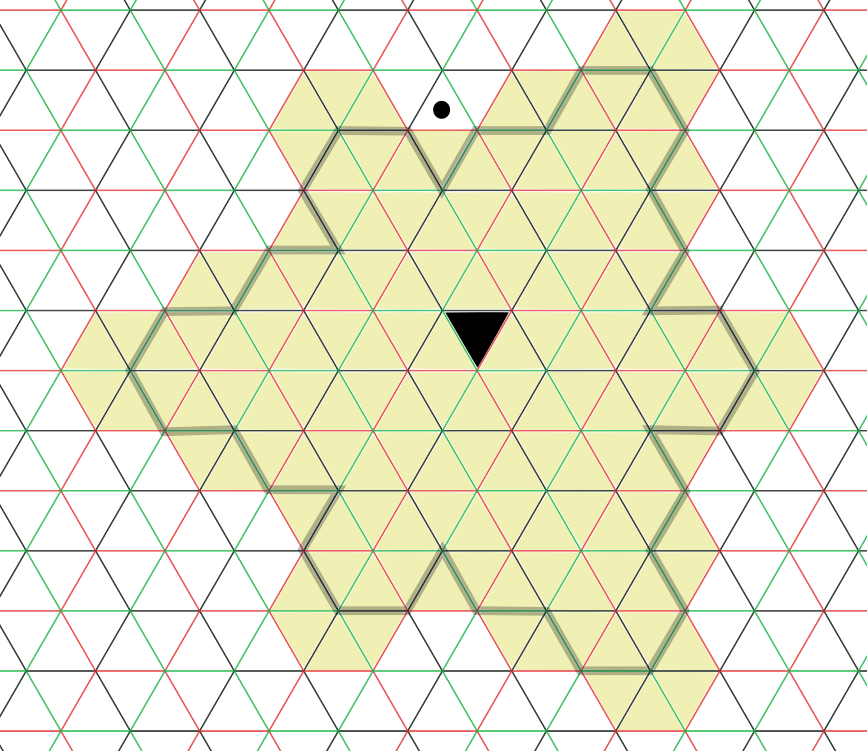}
         \caption{The annex of the alcove $s_0s_2s_1s_0s_2s_0s_1$.}
     \end{subfigure}
     
     \caption{In (a), we have the annex of the alcove $w=s_0s_2s_1s_0s_2s_0$, with $D_R(w)=\{1\}$. In (b), we have the annex of the alcove $ws_1$. We can obtain picture (b) from picture (a) by expanding the set to include every element of the form $y W_{\{0,2\}}$, where $y$ is a shaded alcove in (a). This will include all alcoves from (a), shown in (b) as the highlighted section. }
     \label{figgrid1}
\end{figure}

Lastly, we prove a symmetry of the annex if our alcove is in the fundamental chamber of the Coxeter complex. Recall that an alcove $x$ lies in the fundamental chamber with respect to $W_0$ if and only if $\ell(s_ix)>\ell(x)$ for all $i\neq 0$. 

\begin{proposition}
    \label{symmetry2}
    If $W$ is an affine Coxeter complex, and the alcove $x$ lies in the fundamental chamber, then $\nPSh{x}$ is stable under the left action of $W_{0}=W_{I-\{0\}}$. 
\end{proposition}

\begin{proof}
    This statement is equivalent to saying that, if $w\in W_{0}$ and $z\in \nPSh{x}$, then $wz\in \nPSh{x}$. Note that, as $x$ lies in the fundamental chamber, any reduced word for $x$ starts with $s_0$. 

    Let us first prove that, if $z\in W$ is in the fundamental chamber and $w\in W_0$, then $z\in \nPSh{x}$ if and only if $wz\in \nPSh{x}$. First suppose that $z\in \nPSh{x}$, and that $wz\notin \nPSh{x}$. Then $x\leq wz$. Now $wz$ is reduced as $z$ is minimal in $W_0\cdot z$. But there is no reduced word equal to $x$ beginning with $s_i$ for $i\neq 0$. As any reduced word equal to $w$ only contains $s_i$ for $i\neq 0$ we must have that $x\leq z$. This contradicts the fact that $z\in \nPSh{x}$. Now suppose that $wz\in \nPSh{x}$. It is obvious that $z\leq wz$, as $z$ is minimal in $W_0\cdot z$. So if $z\notin \nPSh{x}$, $x\leq z\leq wz$, contradicting the fact that $wz\in \nPSh{x}$. 

    Now for any $z\in W$, $w\in W_0$, let $z'=w'z$ be the element in $W_0\cdot z$ lying in the fundamental chamber, with $w'\in W_0$. Then we have shown that $z=w^{-1}z'\in \nPSh{x}$ if and only if $z'\in \nPSh{x}$, as $w^{-1}\in W_0$. Then $z'\in \nPSh{x}$ if and only if $wz=(ww')z'\in \nPSh{x}$, as $ww'\in W_0$. 
\end{proof}

\section{The boundary of the annex}\label{3}

Our aim is to now prove Theorem \ref{theorem}, which helps us to define the boundary of an annex.  
To prove this statement, we first need to classify the descent sets of the elements $r_t\cdots r_1ws_i$. We then prove that if there are two adjacent reflections $r_1$ and $r_2$, each increasing the length of $x$ by one each time, then $\ell(r_1r_2r_1x)=\ell(x)\pm 1$. Together, these statements form an inductive argument.  

\subsection{Descent sets}

In this section, we have a series of propositions that establish what is in the descent sets of the elements $r_t\cdots r_1ws_i$.  In the following propositions, we establish that, if the conditions of the theorem hold, then the reflection $r_n$ increases the length of $w$ and $r_{n-1}\cdots r_1 w$. From this, we can conclude that $\ell(r_n\cdots r_1ws_i)<\ell(r_n\cdots r_1w)$. Refer back to Chapter \ref{hyperplanes} for definitions on the halfspaces $H^1$ and $H^\infty$. 

 \begin{proposition}\label{subsets}
    Let $W$ be an affine Coxeter complex. Let $x\in W$. Let $r_1,\hdots,r_n\in W$ be a sequence of adjacent parallel reflections such that \[\ell(r_t\cdots r_1x)>\ell(r_{t-1}\cdots r_1x),\] for all $1\leq t\leq n$. Then $\Hid_{r_{t-1}}\subseteq \Hid_{r_{t}}$ for all $3\leq t\leq n$. 
 \end{proposition}

\begin{proof}
    Let us assume that $\Hid_{r_{t-1}}\not\subseteq \Hid_{r_{t}}$ for some $3\leq t\leq n$. Then we cannot have that $\Hid_{r_{t-1}}\not\subseteq \Hinf_{r_{t}}$, otherwise $\Hid_{r_{t-1}}\cap \Hid_{r_{t}}=\varnothing$, but $1\in\Hid_{r_{t-1}}\cap \Hid_{r_{t}}$. We also cannot have that $\Hinf_{r_{t-1}}\subseteq \Hinf_{r_{t}}$, as then $\Hinf_{r_{t-1}}\cap \Hid_{r_{t}}=\varnothing$, but $r_{t-1}\cdots r_1 x\in \Hinf_{r_{t-1}}\cap \Hid_{r_{t}}$. Therefore, $\Hinf_{r_{t-1}}\subseteq \Hid_{r_{t}}$. 
    
    Now, as $H_{r_{t-2}},H_{r_{t-1}}$ and $H_{r_{t}}$ are adjacent and parallel, we must have that either $\Hinf_{r_{t-2}}\subseteq \Hinf_{r_{t-1}}$ or $\Hid_{r_{t-2}}\subseteq \Hinf_{r_{t-1}}$. If $\Hinf_{r_{t-2}}\subseteq \Hinf_{r_{t-1}}$, then $\Hinf_{r_{t-2}}\cap \Hid_{r_{t-1}}=\varnothing$, so this cannot occur because $r_{t-2}\cdots r_1 x\in \Hinf_{r_{t-2}}\cap \Hid_{r_{t-1}}$. If $\Hid_{r_{t-2}}\subseteq \Hinf_{r_{t-1}}$, then $\Hid_{r_{t-2}}\cap \Hid_{r_{t-1}}=\varnothing$, but this cannot occur as $1\in\Hid_{r_{t-2}}\cap \Hid_{r_{t-1}}$. Therefore, $\Hid_{r_{t-1}}\subseteq \Hid_{r_{t}}$ for $3\leq t\leq n$.
\end{proof}

\begin{proposition}\label{threeparallel}
    Let $W$ be an affine Coxeter complex. Let $x\in W$. Consider three adjacent parallel reflections $r_1, r_2$ and $r_3$ in $W$. Then the actions of $r_3r_2r_1$ and $r_2$ are the same action, and so $r_3r_2r_1=r_2$.
\end{proposition}
\begin{proof}
    Consider the construction of the affine Coxeter complex via the root system. Then, for some positive root $\alpha$, the hyperplanes fixed by $r_1,r_2$ and $r_3$ are $H_{\alpha,m}$, $H_{\alpha,m+1}$ and $H_{\alpha,m+2}$ respectively, for some $m\in\mathbb{Z}$. Then we can explicitly calculate where any vector $x$ is sent to under $r_3r_2r_1$. Let $(\gamma,x)=k$. 
    \begin{align*}
        r_3r_2r_1(x)&= r_3r_2(x-(k-m)\gamma^\vee)\\
        &= r_3(x-(k-m)\gamma^\vee-((\gamma,x-(k-m)\gamma^\vee)-(m+1))\gamma^\vee)\\
        &=r_3(x-(k-m)\gamma^\vee -k\gamma^\vee+2(k-m)\gamma^\vee+(m+1)\gamma^\vee\\
        &=r_3(x+\gamma^\vee)\\
        &=x+\gamma^\vee-((\gamma,x+\gamma^\vee)-(m+2))\gamma^\vee\\
        &=x+\gamma^\vee-k\gamma^\vee-2\gamma^\vee+m\gamma^\vee+2\gamma^\vee\\
        &=x-(k-(m+1))\gamma^\vee\\
        &=r_2(x).
    \end{align*}
    Therefore, $r_3r_2r_1$ and $r_2$ are the same action. Hence, $r_3r_2r_1=r_2$ as $W$ acts faithfully on the complex. 
    \end{proof}

\begin{proposition}\label{n>2containsx}
    Let $W$ be an affine Coxeter complex. Let $x\in W$. Let $r_1,\hdots,r_n\in W$ be a sequence of adjacent parallel reflections such that \[\ell(r_t\cdots r_1x)>\ell(r_{t-1}\cdots r_1x),\] for all $1\leq t\leq n$, with $n\geq 3$. 
    Then $x\in \Hid_{r_t}$ for all $1\leq t\leq n$. 
\end{proposition}
\begin{proof}
    By assumption, $\ell(r_1x)>\ell(x)$, so $x\in \Hid_{r_1}$. Now, by 
    Proposition \ref{subsets}, $\Hid_{r_2}\subseteq \Hid_{r_3}$. Therefore, $\Hinf_{r_3}\subseteq \Hinf_{r_2}$. Now $\ell(r_3r_2r_1x)>\ell(r_2r_1x),$ so $r_3r_2r_1x\in \Hinf_{r_3}$. Then, by Proposition \ref{threeparallel}, $r_3r_2r_1=r_2$, and we have $r_2x\in \Hinf_{r_3}$. Therefore, $r_2x\in \Hinf_{r_2}$, and so $x\in \Hid_{r_2}$. Now, again by Proposition \ref{subsets}, $\Hid_{r_2}\subset \Hid_{r_t}$ for all $3\leq t\leq n$. Hence, $x\in \Hid_{r_t}$ for all $1\leq t\leq n$. 
\end{proof}
\begin{proposition}
    Let $W$ be an affine Coxeter complex. Let $w\in W$ and $i\in D_R(w)$. Let $r_1,\hdots,r_n\in W$ be a sequence of adjacent parallel reflections such that \[\ell(r_t\cdots r_1ws_i)>\ell(r_{t-1}\cdots r_1 ws_i),\] for all $1\leq t\leq n$,
    and $r_1ws_i\neq w$, with $n\geq 3$. Then  $w\in \Hid_{r_n}$. 
\end{proposition}
\begin{proof}
    By Proposition \ref{n>2containsx}, $ws_i\in \Hid_{r_n}$. Now if $w\in \Hinf_{r_n}$, as $\Hinf_{r_n}\subseteq \Hinf_{r_1}$ by Proposition \ref{subsets}, we have that $w\in \Hinf_{r_1}$. But $\ell(r_1ws_i)>\ell(ws_i)$, so $ws_i\in \Hid_{r_1}$. So the $s_i$ panel in $w$ is in the hyperplane $H_{r_1}$. Therefore, $r_1ws_i=w$, which is a contradiction. So $w\in \Hid_{r_n}$.
\end{proof}

\begin{proposition}\label{rn...r1winHid}
    Let $W$ be an affine Coxeter complex. Let $w\in W$ and $i\in D_R(w)$. Let $r_1,\hdots,r_n\in W$ be a sequence of adjacent parallel reflections such that \[\ell(r_t\cdots r_1ws_i)>\ell(r_{t-1}\cdots r_1 ws_i),\] and $r_{t}\cdots r_1ws_i\neq r_{t-1}\cdots r_1w$, for all $1\leq t\leq n$. Then $r_{n-1}\cdots r_1w\in \Hid_{r_n}$. 
\end{proposition}
\begin{proof}
     If $r_{n-1}\cdots r_1w\in \Hid_{r_n}$, then $r_{n-1}\cdots r_1w$ and $r_{n-1}\cdots r_1ws_i$ lie on opposite sides of $H_{r_n}$. So the $s_i$ panel of $r_{n-1}\cdots r_1w$ and $r_{n-1}\cdots r_1ws_i$ lies in $H_{r_n}$. Therefore, the action of $r_n$ sends $r_{n-1}\cdots r_1ws_i$ to $r_{n-1}\cdots r_1w$. But we assumed that $r_{t}\cdots r_1ws_i\neq r_{t-1}\cdots r_1w$. So $r_{n-1}\cdots r_1w\in \Hid_{r_n}$. 
\end{proof}

\begin{corollary}\label{iinrn...r1wsi}

    Let $W$ be an affine Coxeter complex. Let $w\in W$ and $i\in D_R(w)$. Let $r_1,\hdots,r_n\in W$ be a sequence of adjacent parallel reflections such that \[\ell(r_t\cdots r_1ws_i)=\ell(r_{t-1}\cdots r_1 ws_i)+1,\] and $r_{t}\cdots r_1ws_i\neq r_{t-1}\cdots r_1w$, for all $1\leq t\leq n$. Then  $i\notin D_R({r_n\cdots r_1ws_i})$.
\end{corollary}
\begin{proof}
    
    By induction on $n$. By assumption, $i\notin D_R({ws_i})$. 
    
    Now assume that $i\notin D_R({r_{n-1}\cdots r_1ws_i})$. By Proposition \ref{rn...r1winHid}, $r_{n-1}\cdots r_1w\in \Hid_{r_n}$. So therefore,
    \[\ell(r_{n}\cdots r_1w)>\ell(r_{n-1}\cdots r_1w)=\ell(r_{n-1}\cdots r_1ws_i)+1=\ell(r_{n}\cdots r_1ws_i).\]
    Hence, $i\notin D_R({r_n\cdots r_1ws_i})$.
\end{proof}

The next results establish that $\bigcap_{t<n} D_R(r_t\cdots r_1ws_i)$ is nonempty. We do this by first proving that, if $j\in D_R({r_n\cdots r_1ws_i})$, then $j\in D_R({r_t\cdots r_1ws_i})$ for all $1\leq t\leq n$, unless $r_{l}\cdots r_1ws_i= r_{l-1}\cdots r_1ws_is_j$ for some $l$. We then show that, under some conditions of hyperplanes not bounding the identity alcove, if $j\in D_R(x)$, then $j\in D_R(r_2r_1x)$. This establishes that, under these conditions, if $j\in D_R(x)\cap D_R(r_1x)$, then $j\in\bigcap_{t\leq n} D_R(r_t\cdots r_1ws_i)$. We then show that this condition must hold for $t<n$, and so $\bigcap_{t<n} D_R(r_t\cdots r_1ws_i)$ is nonempty.

\begin{proposition}\label{shareddecreasing}
    Let $W$ be an affine Coxeter complex. Let $w\in W$ and $i\in D_R(w)$. Let $r_1,\hdots,r_n\in W$ be a sequence of adjacent parallel reflections such that \[\ell(r_t\cdots r_1ws_i)=\ell(r_{t-1}\cdots r_1 ws_i)+1,\] for all $1\leq t\leq n$. Let $j\in D_R({r_n\cdots r_1ws_i})$. If $r_{t}\cdots r_1ws_i\neq r_{t-1}\cdots r_1ws_is_j$ for all $1\leq t\leq n$, then $j\in D_R({r_t\cdots r_1ws_i})$ for all $1\leq t\leq n$.
\end{proposition}
\begin{proof}
    We will show that $j\in D_R({r_{n-1}\cdots r_1ws_i})$. Then we will be able to immediately conclude that $j\in D_R({r_t\cdots r_1ws_i})$ for all $1\leq t\leq n$.

    By assumption, $r_{n-1}\cdots r_1ws_i\in \Hid_{r_n}$. Now also note that, if $r_{n-1}\cdots r_1ws_is_j\in \Hinf_{r_n}$, then the action of $r_n$ sends $r_{n-1}\cdots r_1ws_i$ to $r_{n-1}\cdots r_1ws_is_j$. But we assumed that this is not true, so $r_{n-1}\cdots r_1ws_is_j\in \Hid_{r_n}$. Now
    \[\ell(r_{n-1}\cdots r_1ws_i)=\ell(r_n\cdots r_1ws_i)-1=\ell(r_n\cdots r_1 ws_is_j)>\ell(r_{n-1}\cdots r_1 ws_is_j).\]
    So $j\in D_R({r_{n-1}\cdots r_1ws_i})$.  
\end{proof}
\begin{lemma}\label{r2r1H3parallel}
    Consider two parallel hyperplanes $H_1$, $H_2$ in Euclidean space, and let $r_1$ and $r_2$ be their associated reflections. Let $H_3$ be a third hyperplane. Then $r_2r_1H_3$ is parallel to $H_3$. 
\end{lemma}

\begin{proof}
    Let $H_1=H_{\gamma,n_1}$, $H_2=H_{\gamma,n_2}$ and $H_3=H_{\alpha,m}$, for some vectors $\gamma$ and $\alpha$. Let $v\in H_3$ be any vector lying on $H_3$. Let us calculate $r_2r_1v$.
    \begin{align*}
        r_2r_1v&=r_2(v-((v,\gamma)-n_1)\gamma^\vee)\\
        &=v-((v,\gamma)-n_1)\gamma^\vee-(v-((v,\gamma)-n_1)\gamma^\vee,\gamma)-n_2)\gamma^\vee\\
        &=v+(n_2-n_1)\gamma^\vee.
    \end{align*}
    Now 
    \[(r_2r_1v,\alpha)=(v+(n_2-n_1)\gamma^\vee,\alpha)=m+(n_2-n_1)(\gamma^\vee,\alpha).\]
    Now $k=(\gamma^\vee,\alpha)\in\mathbb{Z}$, by the axioms of a root system. So $r_2r_1v\in H_{\alpha,m+(n_2-n_1)k}$, for all $v\in H_3$. So $r_2r_1H_3=H_{\alpha,m+(n_2-n_1)k}$, which is parallel to $H_3$. 
\end{proof}
\begin{proposition}\label{jnotinJr2r1x}
    Let $x\in W$. Let $H_1,H_2$ be parallel hyperplanes, and let $r_1$ and $r_2$ be their associated reflections. Assume that $j\in D_R({x})$. Let $H_3=H_{xs_jx^{-1}}$ and $H_4=H_{r_2r_1xs_jx^{-1}r_1r_2}$ and assume that the identity alcove is not bounded between $H_3$ and $H_4$. Then $j\in D_R({r_2r_1x})$. 
\end{proposition}

\begin{proof}
    By Lemma \ref{r2r1H3parallel}, $H_3$ and $H_4$ are parallel. Then $\Hid_3\not\subseteq\Hinf_4$, as then $\Hid_3\cap\Hid_4=\varnothing$, which is a contradiction. Also $\Hinf_3\not\subseteq\Hid_4$, as then the identity alcove would be bounded between $H_3$ and $H_4$. So either $\Hinf_3\subseteq \Hinf_4$ or $\Hid_3\subseteq \Hid_4$. 

    Now $r_2r_1$ is the combination of two parallel reflections, so it is a translation. It sends $H_3$ to $H_4$. As either $\Hinf_3\subseteq \Hinf_4$ or $\Hid_3\subseteq \Hid_4$, 
    \[r_2r_1\cdot \Hid_3=\Hid_4\quad \text{and}\quad r_2r_1\cdot \Hinf_3=\Hinf_4.\] 
    Then, as $x\in \Hinf_3$, $r_2r_1x\in\Hinf_4$. Hence, $r_2r_1xs_j\in \Hid_4$, and so $j\in D_R({r_2r_1x})$.
\end{proof}

\begin{corollary}
    Let $W$ be an affine Coxeter complex. Let $x\in W$. Let $r_1,\hdots,r_n\in W$ be a sequence of adjacent parallel reflections such that \[\ell(r_t\cdots r_1x)=\ell(r_{t-1}\cdots r_1x)+1,\] for all $1\leq t\leq n$. Assume there exists a $j\in D_R({x})\cap D_R({r_1x})$. Let $H_t$ be the hyperplane separating $r_t\cdots r_1x$ and $r_t\cdots r_1xs_j$. Assume no two $H_t$ bound the origin. Then $j\in D_R({r_t\cdots r_1x})$ for all $1\leq t\leq n$.
\end{corollary}

\begin{proposition}\label{nonemptyor}
   Let $W$ be a rank 2 affine Coxeter complex. Let $w\in W$ and $i\in D_R(w)$. Let $r_1,\hdots,r_n\in W$ be a sequence of adjacent parallel reflections such that \[\ell(r_t\cdots r_1ws_i)=\ell(r_{t-1}\cdots r_1 ws_i)+1,\] 
    and $r_{t}\cdots r_1ws_i\neq r_{t-1}\cdots r_1w$, for all $1\leq t\leq n$. Suppose there are no hyperplanes parallel to $H_{r_1}$ between $w$ and $H_{r_1}$. Either $\bigcap_t D_R(r_t\cdots r_1ws_i)$ is nonempty, or there exists $j\in D_R(r_n\cdots r_1ws_i)$ such that $r_n\cdots r_1ws_i=r_{n-1}\cdots r_1ws_is_j$. 
\end{proposition}

\begin{proof}
    Suppose $\bigcap_t D_R(r_t\cdots r_1ws_i)$ is empty, and we cannot find $j\in D_R(r_n\cdots r_1ws_i)$ such that $r_n\cdots r_1ws_i=r_{n-1}\cdots r_1ws_is_j$. Let $k\in D_R(r_n\cdots r_1 ws_i)$. By Corollary  \ref{iinrn...r1wsi}, $i\notin D_R(r_n\cdots r_1ws_i)$, so $s_k\neq s_i$. Now, if at some point $r_k\cdots r_1 ws_i=r_{k-1}\cdots r_1 ws_is_k$, the panels of type $s_k$ in all $r_t\cdots r_1 ws_i$ must be parallel to the reflecting hyperplanes. So for $r_n\cdots r_1 ws_i$, its $s_k$ panel must either lie on $H_{r_n}$, which cannot occur as then $r_n\cdots r_1ws_i=r_{n-1}\cdots r_1ws_is_k$, or must lie on the next parallel line, but then $j\notin D_R(r_n\cdots r_1ws_i),$ which also cannot occur. So then, $r_k\cdots r_1 ws_i\neq r_{k-1}\cdots r_1 ws_is_k$ for any $t$, so by Proposition \ref{shareddecreasing}, $j\in D_R(r_t\cdots r_1ws_i)$ for all $t$.
\end{proof}

\begin{corollary}\label{nonemptyor2}
    Suppose the conditions of Proposition \ref{nonemptyor} hold on $w$, $s_i$ and $r_1,\hdots,r_n$. Then the intersection $\bigcap_{t<n} D_R(r_t\cdots r_1ws_i)$ is nonempty. 
\end{corollary}

\begin{proof}
    If it is empty, then $\bigcap_t D_R(r_t\cdots r_1ws_i)$ is empty. By Proposition \ref{nonemptyor}, there is a $j\in D_R(r_n\cdots r_1ws_i)$ such that $r_n\cdots r_1ws_i=r_{n-1}\cdots r_1ws_is_j$. But, also by Proposition \ref{nonemptyor}, there is a $k\in D_R(r_{n-1}\cdots r_1ws_i)$ such that $r_{n-1}\cdots r_1ws_i=r_{n-2}\cdots r_1ws_is_k$. Now the alcove $r_{n-1}\cdots r_1ws_i$ has two panels, $s_j$ and $s_k$, lying on adjacent parallel hyperplanes. But this is not possible, as $r_{n-1}\cdots r_1ws_i$ is a simplex. So $\bigcap_{t<n} D_R(r_t\cdots r_1ws_i)$ is nonempty. 
\end{proof}

We can now establish that either $\bigcap_{t\leq n} D_R(r_t\cdots r_1ws_i)$ is nonempty, or $r_{n-1}\cdots r_1ws_i\in \nPSh{w}$ implies that $r_n\cdots r_1ws_i\in\nPSh{w}$. This means that, when proving the main theorem, we can assume that all $r_t\cdots r_1ws_i$ share a common decreasing panel. 
\begin{lemma}
    Suppose the conditions of Proposition \ref{nonemptyor} hold on $w$, $s_i$ and $r_1,\hdots,r_n$. Suppose $H_{r_0}$ and $H_{r_1}$ bound the identity. If there exists $j\in D_R(r_n\cdots r_1ws_i)$ such that $r_n\cdots r_1ws_i=r_{n-1}\cdots r_1ws_is_j$, then $j\notin D_R(w)$.
\end{lemma}

\begin{proof}
    Note that the $s_j$ type panel of $ws_i$ must lie either on the $H_{r_0}$ hyperplane or the $H_{r_1}$ hyperplane. Therefore, as $H_{r_0}$ and $H_{r_1}$ bound the identity, $j\notin D_R(ws_i)$. So $i,j\notin D_R(ws_i)$, and hence $ws_i$ is minimal in the coset $ws_iW_{\{s_i,s_j\}}$. Therefore, $w$ is not maximal in this coset, and as $i\in D_R(w)$, we must have that $j\notin D_R(w)$. 
\end{proof}

\begin{corollary}\label{assume}
    Suppose the conditions of Proposition \ref{nonemptyor} hold on $w$, $s_i$ and $r_1,\hdots,r_n$. Suppose $H_{r_0}$ and $H_{r_1}$ bound the identity.
    If $r_{n-1}\cdots r_1ws_i\in \nPSh{w},$ and there exists $j\in D_R(r_n\cdots r_1ws_i)$ such that $r_n\cdots r_1ws_i=r_{n-1}\cdots r_1ws_is_j$, then $r_n\cdots r_1ws_i\in\nPSh{w}$. 
\end{corollary}  

\begin{corollary}
    Suppose the conditions of Proposition \ref{nonemptyor} hold on $w$, $s_i$ and $r_1,\hdots,r_n$. Suppose $H_{r_0}$ and $H_{r_1}$ bound the identity.
    If $r_{n-1}\cdots r_1ws_i\in \nPSh{w},$ and $\bigcap_{t\leq n} D_R(r_t\cdots r_1ws_i)$ is empty, then $r_n\cdots r_1ws_i\in\nPSh{w}$. 
\end{corollary}

\begin{proof}
    Proposition \ref{nonemptyor} establishes that \[r_n\cdots r_1ws_i=r_{n-1}\cdots r_1ws_is_j\] for some $k\in D_R(r_{n-1}\cdots r_1ws_i)$. The result follows from Corollary \ref{assume}. 
\end{proof}

\subsection{Adjacent parallel reflections}

We now want to prove that when we have three adjacent parallel reflections $H_{r_0}$, $H_{r_1}$ and $H_{r_2}$, such that 
\[\ell(r_2r_1x)=\ell(r_1x)+1=\ell(x)+2,\]
and $H_{r_0}$ and $H_{r_1}$ bound $x$, then $\ell(r_0x)=\ell(x)\pm 1$. We prove this as Theorem \ref{pm1}. This theorem will then allow us to form an inductive argument, so that we can always assume the identity lies between $H_{r_0}$ and $H_{r_1}$. We first show that, if the reflection $r_0$ increases the length of $x$, then it can only increase the length by 1.

\begin{lemma}
    Suppose we have three adjacent parallel reflections $H_{r_0}$, $H_{r_1}$ and $H_{r_2}$, such that 
\[\ell(r_2r_1x)=\ell(r_1x)+1=\ell(x)+2,\]
and $\ell(r_0x)>\ell(x)$. Suppose $x$ is bounded between $H_{r_0}$ and $H_{r_1}$. Then $\ell(r_0x)=\ell(x)+1$.  
\end{lemma}

\begin{proof}
    As $\ell(r_0x)>\ell(x)$ and $\ell(r_1x)>\ell(x)$, the identity alcove must lie between $H_{r_0}$ and $H_{r_1}$. So $r_0x\in H_{r_1}^1$. Now $r_2r_1=r_1r_0$, so \[\ell(r_0x)<\ell(r_1r_0x)=\ell(r_2r_1x)=\ell(x)+2,\] so $\ell(x)<\ell(r_0x)<\ell(x)+2$. Therefore, $\ell(r_0x)=\ell(x)+1$. 
\end{proof}
So the statement is true if $\ell(r_0x)>\ell(x)$. So now we will assume the hypotheses in the claim hold and that $\ell(r_0x)<\ell(x)$. We then show that, if $r_1x$ shares a panel with $H_{r_2}$, then the claim holds. 

\begin{lemma}
    Suppose we have three adjacent parallel reflections $H_{r_0}$, $H_{r_1}$ and $H_{r_2}$, such that 
\[\ell(r_2r_1x)=\ell(r_1x)+1=\ell(x)+2,\]
and $\ell(r_0x)<\ell(x)$. Suppose $x$ is bounded between $H_{r_0}$ and $H_{r_1}$. Suppose there is a generator $s_i$ such that $r_1xs_i=r_2r_1x$. Then $\ell(r_0x)=\ell(x)-1$. 
\end{lemma}

\begin{proof}
    We know that $r_2r_1=r_1r_0$, so $r_1xs_i=r_1r_0x$. Therefore, $xs_i=r_0x$, and so $\ell(r_0x)=\ell(x)\pm 1$. Hence, as $\ell(r_0x)<\ell(x)$, $\ell(r_0x)=\ell(x)-1$. 
\end{proof}

We next prove a few geometric statements that will help us conclude that $r_0x$, $x$, $r_1x$ and $r_2r_1x$ have some shared descent types. This will help us form our inductive argument.

\begin{lemma}\label{triangleangles}
    Let $\Delta ABC$ be a triangle such that each angle is at most $\pi/2$. Let $L$ be the line through $A$ and $B$, and let $B'$ be another point on $L$ lying on the same side of $A$ as $B$, and further away from $A$. Then the angle at $B'$ in the triangle $\Delta AB'C$ is still at most $\pi/2$. 

\begin{center}

\begin{tikzpicture}[scale=0.6]
\coordinate (A)  at (0,0);
\coordinate (B)  at (4,-1);
\coordinate (Bp) at (6,-1.5);
\coordinate (C) at (4,2);
\draw ($(A)!-0.3!(Bp)$) -- ($(A)!1.3!(Bp)$) node[right] {$L$};
\draw (A) -- (C);
\draw (C) -- (Bp);
\draw (C) -- (B);
\draw pic[
  draw,
  angle radius=0.5cm,
  "$\theta$"
] {angle = C--Bp--B};
\fill (A)  circle (1.5pt);
\fill (B)  circle (1.5pt);
\fill (Bp) circle (1.5pt);
\fill (C) circle (1.5pt);
\node[below]        at (A)  {$A$};
\node[below left]  at (B)  {$B$};
\node[below ] at (Bp) {$B'$};
\node[above]       at (C)  {$C$};
\end{tikzpicture}
\end{center}
    
\end{lemma}

\begin{proof}
    Let $\gamma=\vec{AB}$, and let $\delta=\vec{AC}$. Now $(\gamma,\delta)\geq 0$, as the angle at $A$ is at most $\pi/2$. Similarly, $(\gamma,\gamma-\delta)\geq 0$ and $(\delta,\delta-\gamma)\geq 0$. Now $\vec{AB'}=k\gamma$ for $k>1$. Now, $(\gamma,\gamma-\delta)\geq 0$ implies that $(\gamma,\gamma)\geq (\gamma,\delta)$. Therefore, for $k>1$, $k^2(\gamma,\gamma)\geq k(\gamma,\delta).$ Now the vector between $B'$ and $A$ is $-k\gamma$, and the vector between $B'$ and $C$ is $\delta-k\gamma$. Now 
    \[(-k\gamma,\delta-k\gamma)=(k\gamma,k\gamma-\delta)=k^2(\gamma,\gamma)-k(\gamma,\delta)\geq 0.\]
    Hence, the angle at $B'$ must be at most $\pi/2$.
\end{proof}

\begin{proposition}
    Let an alcove $x$ lie between $H_{r_0}$ and $H_{r_1}$, with vertices $A,B,C$. Suppose $C$ is in the orbit of the origin and lies on $H_{r_1}$. Let $H=H_{xs_0x^{-1}}$. Then the angle of the sector defined by $H_{r_0}$ and $H$ containing $x$ is at most $\pi/2$. 
\end{proposition}

\begin{proof}
    Note that $A$ and $B$ lie between $H_{r_0}$ and $H_{r_1}$. Let $L$ be the line containing $A$ and $B$. So now $A$ and $B$ lie on the same side of $H_{r_1}$ on $L$. Assume without loss of generality that $B$ is closer than $A$ to $H_{r_1}$ along $L$. Now any interior angle of $x$ is at most $\pi/2$. So, by Lemma \ref{triangleangles}, if we shift $B$ along $L$ onto $H_{r_1}$, the angle at this point must still be at most $\pi/2$. This is exactly the angle of the sector defined by $H_{r_0}$ and $H$ containing $x$.
\end{proof}

\begin{corollary}
    Suppose $x$ is bounded between $H_{r_0}$ and $H_{r_1}$, and the vertex of $x$ in the orbit of the origin lies on $H_{r_1}$. Let $H_{r_1}=H_{\alpha,n}$, $H_{r_2}=H_{\alpha,n+1}$ such that $n\geq 0$, and let $H=H_{xs_0x^{-1}}=H_{\beta,m}$ such that $m\geq 0$. Then $(\alpha,\beta)\geq 0$.
\end{corollary}

We are now able to prove, in the next few statements, that if no simplex of $x$ lies on $H_{r_0}$ and $0\in D_R(x)$, then $s_0$ is also in the right descent sets of $r_0x$, $r_1x$ and $r_2r_1x$. 
\begin{lemma} \label{angle<pi/2}
    Let $r_0,r_1$ and $r_2$ be parallel, adjacent reflections. Suppose $x$ is bounded between $H_{r_0}$ and $H_{r_1}$. Suppose $H=H_{xs_ix^{-1}}$ and angle of the sector defined by $H_{r_0}$ and $H$ containing $x$ is at most $\pi/2$. Then $i\in D_R(r_2r_1x)$.
\end{lemma}
\begin{proof}
    Let $H_{r_1}=H_{\alpha,n}$, $H_{r_2}=H_{\alpha,n+1}$ such that $n\geq 0$, and let $H=H_{xs_0x^{-1}}=H_{\beta,m}$ such that $m\geq 0$. Now, for any vertex $v$, $r_2r_1(v)=v+\alpha^\vee$. Now let $v$ be any vertex on $H$. Now, $(\beta,\alpha^\vee)\geq 0$. Then $r_2r_1(v)$ satisfies
    \[(\beta,r_2r_1(v))=(\beta, v+\alpha^\vee)= m+(\beta,\alpha^\vee),\]
    and so $r_2r_1H=H_{\beta,m+(\beta,\alpha^\vee)},$ with $m+(\beta,\alpha^\vee)>m.$
    
    Now let $v'$ be the vertex of $x$ in the orbit of the origin. Now $v'$ does not lie on $H$, and $0\in D_R(x)$, so $H$ must separate the origin and $v'$. As $m>0$, this means that $(\beta,v')>m$. Then,
    \[(\beta, r_2r_1(v'))=(\beta,v'+\alpha^\vee)>m+(\beta,\alpha^\vee). \]
    So $r_2r_1(v')\in r_2r_1H^\infty$. Therefore, $i\in D_R(r_2r_1x)$. 
\end{proof}

\begin{corollary}\label{s0inxs0inr2r1x}
    Let $r_0,r_1$ and $r_2$ be parallel, adjacent reflections. Suppose $x$ is bounded between $H_{r_0}$ and $H_{r_1}$, and the vertex of $x$ in the orbit of the origin lies on $H_{r_1}$. Then, if $0\in D_R(x)$, $0\in D_R(r_2r_1x)$. 
\end{corollary}

\begin{lemma}\label{angle>pi/2}
    Let $H_r=H_{\alpha,n}$ and $H_{\beta,m}$ be two nonparallel hyperplanes such that $n\geq0$ and $m\geq0$. Suppose that $(\alpha,\beta)<0$. If $v$ satisfies $(\alpha,v)>n$ and $(\beta ,v)>m$, then either the origin lies on all three hyperplanes $H_r$, $H_{\beta,m}$ and $rH_{\beta,m}$, or $rv$ is separated from the origin by $rH_{\beta,m}$. 
\end{lemma}
\begin{proof}
    First note that $rH_{\beta,m}-H_{\beta-(\alpha^\vee,\beta)\alpha,t}$ for some $t$. Now consider the vertex $v'$ lying on both $H_r$ and $rH_{\beta,m}$. Now $(\alpha,v)=n$ and $(\beta,v)=m$. Then,
    \[t=(\beta-(\alpha^\vee,\beta)\alpha,v')=m-(\alpha^\vee,\beta)n.\]
    Note that $(\alpha^\vee,\beta)< 0$. So $t\geq 0$. First suppose that $t=0$. Then $m=0$ and  $n=0$. Then the origin lies on both $H_r$ and $H_{\beta,m}$, and so it must lie on all three lines.
    
    Now suppose $t>0$, and let $v$ satisfy $(\alpha,v)>n$ and $(\beta,v)>m$. Then,
    \begin{align*}  
    (\beta-(\alpha^\vee,\beta)\alpha),rv)&=(\beta-(\alpha^\vee,\beta)\alpha,v-((\alpha,v)-n)\alpha^\vee)\\&=(\beta,v)-(\alpha^\vee,\beta)(\alpha,v)-[(\alpha^\vee,v)-n][(\alpha^\vee,\beta)-(\alpha^\vee,\beta)(\alpha^\vee,\alpha)]\\&=(\beta,v)-(\alpha^\vee,\beta)(\alpha,v)+[(\alpha,v)-n](\alpha^\vee,\beta)\\&=(\beta,v)-(\alpha^\vee,\beta)n\\&>m-(\beta^\vee,\alpha)n=t.
    \end{align*}
    So, as $t>0$, $H_{\beta-(\alpha^\vee,\beta)\alpha,t}$ must separate $rv$ and the origin. 
\end{proof}

\begin{lemma}\label{sharedxr1x}
    Suppose $i\in D_R(x)$ is a shared decreasing panel type of $x$ and $r_1x$. Then $i\in D_R(r_0x)$. 
\end{lemma}
\begin{proof}
    Suppose $H=H_{xs_ix^{-1}}=H_{\beta,m}$, with $m\geq 0$, and $H_{r_0}=H_{\alpha,n}$, with $n\geq 0$. If $(\alpha,\beta)\geq 0$, then $i\in D_R(r_0x)$, by Lemma \ref{angle<pi/2}. So suppose that $(\alpha,\beta)<0$. Consider the vertex $v$ of $x$ not lying on $H$. Then $(\alpha,v)>n$ and $(\beta,v)>m$. By Lemma \ref{angle>pi/2}, either the origin lies on the intersection of $H_{r_0}$ and $H$, or $r_0H$ separates $rv$ and the origin. Suppose the first case is true. Then $m=n=0$. Take another vertex $v'$ of the identity alcove. Then $(\alpha, v')\leq 0 $ and $(\beta, v')\leq 0$. Now $(\beta-(\alpha^\vee,\beta)\alpha, v')\leq 0$. If  $(\beta-(\alpha^\vee,\beta)\alpha, v')= 0$, then $(\alpha,v')=(\beta,v')=0$. But this is a contradiction, as the intersection of $H$ and $H_{r_0}$ is the origin. Therefore, $(\beta-(\alpha^\vee,\beta)\alpha, v')< 0$, and so $r_0H$ separates $v$ from $v'$. Hence, $r_0H$ separates $r_0x$ from the origin, and so $i\in D_R(r_0x)$.
\end{proof}
\begin{proposition}\label{s0descent}
    Suppose $x\cap H_{r_0}= \varnothing$. If $0\in D_R(x)$, then $0\in D_R(r_0x)\cap D_R(r_1x)\cap D_R(r_2r_1x)$.
\end{proposition}
\begin{proof}
    If $x\cap H_{r_0}= \varnothing$, then the vertex $v$ of $x$ in the orbit of the origin must lie on $H_{r_1}$. So, by Corollary \ref{s0inxs0inr2r1x}, $s_0\in r_2r_1x$. Now suppose $0\notin D_R(r_1x)$. If $r_2r_1x=r_1xs_0$, then, as $r_2r_1=r_1r_0$, $r_0x=xs_0$. But this contradicts $x\cap H_{r_0}= \varnothing$. So $r_1x$ and $r_1xs_0$ must lie on the same side of $H_{r_1}$. So $r_1xs_0\in H_{r_2}^1$. Now,
    \[\ell(r_1xs_0)=\ell(r_1x)+1=\ell(r_2r_1x)=\ell(r_2r_1xs_0)+1.\]
    But this contradicts $r_1xs_0\in H_{r_2}^1$. So $0\in D_R(r_1x)$. 

    Now we want to apply Corollary \ref{s0inxs0inr2r1x}, starting in $r_1x$, and reflecting by $r_1$ and then $r_0$. Now $r_1x$ is bounded by $H_{r_2}$ and $H_{r_1}$, and the vertex of $r_1x$ in the orbit of the origin lies on $H_{r_1}$. Also, we have just established that $0\in D_R(r_1x)$. Therefore, by Corollary \ref{s0inxs0inr2r1x}, $0\in D_R(r_0r_1r_1x)=D_R(r_0x)$. 
\end{proof}

Here we prove some elementary results related to the normals of the hyperplanes we are considering. Note that, if $\alpha_1$ and $\alpha_2$ are two linearly independent vectors, with $\alpha'_i$ such that $(\alpha'_i,\alpha_i)=0$ and $(\alpha'_i,\alpha_j)\geq 0$, then $\beta$ lies in the non-reflex angle between $\alpha_1$ and $\alpha_2$ if and only if $(\alpha'_i,\beta)\geq 0$ for $i=1,2$. 

\begin{lemma}
    Let $\Phi$ be a rank 2 root system. Let $\alpha, \gamma$ be two roots forming a root basis. Let $\delta$ be another vector lying between $\alpha$ and $\gamma$. Then $\delta=t_1\alpha+t_2\gamma$ with $t_1,t_2\geq 0$.  
\end{lemma}
\begin{proof}
    As $\alpha, \gamma$ form a root basis, $(\alpha,\gamma)\leq 0$. Let $\alpha'$ be the vector perpendicular to $\alpha$ such that $(\alpha',\gamma)\geq 0$. Then $(\alpha',\delta)\geq 0$. Now $\delta=t_1\alpha+t_2\gamma$, so
    \[0\leq (\delta,\alpha')=t_2(\gamma,\alpha'),\]
    and hence $t_2\geq 0$. Similarly, $0\leq (\delta,\beta')=t_1(\gamma,\beta')$, and so $t_1\geq 0$. 
\end{proof}

\begin{lemma}
    Let $x$ be an alcove, and suppose $x$ is a minimal length element in $xW_0=xW_{\{1,2\}}$. Let $H_{\alpha,n}$, $H_{xs_0x^{-1}}=H_{\beta,m}$ and $H_r=H_{\gamma,t}$ be the three bounding hyperplanes of $x$ with $n,m,t\geq 0$. Suppose $rH_{\beta,m}=H_{\delta,k}$. Then $\delta$ lies between $\alpha$ and $\gamma$. 
     \begin{center}
  \begin{tikzpicture}[scale=0.6, >=stealth]
\coordinate (L1a) at (-3,0.6);
\coordinate (L1b) at (3,-1);

\coordinate (L2a) at (2,-2);
\coordinate (L2b) at (2,4);

\coordinate (L3a) at (-3,-2);
\coordinate (L3b) at (2,3);
\draw[red!35, line width=2pt] (L1a)--(L1b);
\draw[red!35, line width=2pt] (L2a)--(L2b);
\draw[red!35, line width=2pt] (L3a)--(L3b);
\coordinate (A) at (intersection of L1a--L1b and L3a--L3b);
\coordinate (B) at (intersection of L1a--L1b and L2a--L2b);
\coordinate (C) at (intersection of L2a--L2b and L3a--L3b);

\draw[line width=1pt] (A)--(B)--(C)--cycle;

\draw[->]
  ($(A)!0.6!(B)$)
  -- ++(0.25,0.6)
  node[right] {$\beta$};
\draw[->]
  ($(B)!0.6!(C)$)
  -- ++(0.8,0)
  node[right] {$\alpha$};

\draw[->]
  ($(A)!0.6!(C)$)
  -- ++(-0.6,0.5)
  node[left] {$\gamma$};

\node[left]  at (-2.8,-1.6) {$H_{\gamma,t}$};
\node[right] at (3,-1) {$H_{\beta,m}$};
\node[below] at (2,-2) {$H_{\alpha,n}$};

\end{tikzpicture}

    \end{center}
\end{lemma}

\begin{proof}    
    Note that $\beta$ lies between $\alpha$ and $\gamma$. Let $\theta_1$ be the angle between $\alpha$ and $\beta$, and let $\theta_2$ be the angle between $\beta$ and $\gamma$. Now, as $H_{\beta,m}=H_{xs_0x^{-1}}$, the intersections of $H_{\beta,m}$ with $H_{\alpha,n}$ and of $H_{\beta,m}$ with $H_r=H_{\gamma,t}$ cannot be in the orbit of the origin. Therefore, $\pi/3\leq \theta_1\leq \pi/2$, and $\pi/3\leq \theta_2\leq \pi/2$, by observing the possible angles in rank 2 Coxeter complexes.  

    Now let $\gamma'$ be a vector perpendicular to $\gamma$ such that $(\gamma',\alpha)\geq 0$. Then $(\gamma',\beta)\geq 0$. Now, as $\theta_2\leq \pi/2$, $\beta$ must lie between $\gamma$ and $\gamma'$. As $\theta_2\geq \pi/3$, the angle $\theta_3$ between $\gamma'$ and $\beta$ is $0\leq \theta_3\leq \pi/6$. Now $\delta$ is the reflection of $\beta$ over the hyperplane containing $\gamma'$. So the angle between $\gamma'$ and $\delta$ is equal to $\theta_3$, and so is at most $\pi/6.$ So therefore, as the angle between $\gamma$ and $\gamma'$ is $\pi/2$, and the angle between $\gamma'$ and $\alpha$ is at least $\pi/6$, $\delta$ must lie between $\alpha$ and $\gamma$.
       \begin{center}
    \begin{tikzpicture}[scale=3,>=stealth]
\coordinate (O) at (0,0);

\def\alphaAngle{0}
\def\gammaAngle{150}        
\def\gammaPrimeAngle{60}    
\def\thetaa{20}              
\def\deltaAngle{40}        
\def\betaAngle{80}         
\draw[->] (O) -- ({1.2*cos(\alphaAngle)},{1.2*sin(\alphaAngle)}) node[right] {$\alpha$};
\draw[->] (O) -- ({1.0*cos(\deltaAngle)},{1.0*sin(\deltaAngle)}) node[right] {$\delta$};
\draw[->] (O) -- ({1.0*cos(\gammaPrimeAngle)},{1.0*sin(\gammaPrimeAngle)}) node[above right] {$\gamma'$};
\draw[->] (O) -- ({1.0*cos(\betaAngle)},{1.0*sin(\betaAngle)}) node[above] {$\beta$};
\draw[->] (O) -- ({1.0*cos(\gammaAngle)},{1.0*sin(\gammaAngle)}) node[left] {$\gamma$};
\draw ({0.28*cos(\deltaAngle)},{0.28*sin(\deltaAngle)})
      arc[start angle=\deltaAngle, end angle=\gammaPrimeAngle, radius=0.28];
\draw ({0.36*cos(\gammaPrimeAngle)},{0.36*sin(\gammaPrimeAngle)})
      arc[start angle=\gammaPrimeAngle, end angle=\betaAngle, radius=0.36];
\node at ({0.42*cos((\deltaAngle+\gammaPrimeAngle)/2)},
          {0.42*sin((\deltaAngle+\gammaPrimeAngle)/2)}) {$\theta$};
\node at ({0.52*cos((\gammaPrimeAngle+\betaAngle)/2)},
          {0.52*sin((\gammaPrimeAngle+\betaAngle)/2)}) {$\theta$};
\end{tikzpicture}
    \end{center}
\end{proof}

\begin{definition}
    An element $x\in W$ is \ix{preminimal} in $xW_0$ if it is adjacent to the minimal element in $xW_0$.
\end{definition} Using these facts, we can now show that, under some circumstances, a preminimal element must contain $s_0$ in its descent set. 

\begin{proposition} \label{preminimal}
    Suppose $x\cap H_{r_0}= \varnothing$, and suppose $x$ is preminimal in $xW_0$. Suppose further that the minimal element $y$ shares a panel with $H_{r_1}$. Then $0\in D_R(x)$. 
\end{proposition}

\begin{proof}
     Let $H_{r_0}=H_{\alpha,n+1}$, $H_{ys_0y^{-1}}=H_{\beta,m}$ and $H_r=H_{\gamma,t}$ be the three bounding hyperplanes of $y$ with $n,m,t\geq 0$. Let $v$ be the vertex of $y$ not lying on $H_{\beta,m}$. Then $(\beta,v)>m$, $(\alpha,v)=n+1$, and $(\gamma,v)=t$. Now $H_{xs_0x^{-1}}=rH_{ys_0y^{-1}}=H_{\delta,k}$, with $\delta=\beta-(\gamma^\vee,\delta)\gamma$. Let $v'$ be the vertex lying on $H_{\beta,m}$ and $H_{\gamma,t}$. This also lies in $H_{\delta,k}$. Therefore,
     \[k=(\delta,v')=(\beta,v')-(\gamma^\vee,\delta)(\gamma,v')=m-(\gamma^\vee,\delta)t.\]
    Now if $m-(\gamma^\vee,\delta)t>0$, we are done as $(\delta,v)>m-(\gamma^\vee,\delta)t$. So suppose $m-(\gamma^\vee,\delta)t<0$. Now $\delta$ lies between $\alpha$ and $\gamma$, which form a root basis. So $\delta=k_1\alpha+k_2\gamma$, for $k_1,k_2\geq 0$. Then
    \[0>(\delta,v')=k_1(\alpha,v')+k_2(\gamma,v'),\]
    which is a contradiction as $(\alpha,v'),(\gamma,v')\geq 0$. 

    Lastly, suppose $m-(\gamma^\vee,\delta)t=0$. Then $k_1(\alpha,v')+k_2(\gamma,v')=0$. We have that \[k_1,k_2,(\alpha,v'),(\gamma,v')\geq 0.\] First note that $(\alpha,v')>0$, and so we must have that $k_1=0$. Now, if $k_2=0$, $v'$ is the origin. If $(\gamma,v')=0$, the origin must lie on both $H_{\gamma,t}$ and $H_{\delta,k}$. But the intersection of these two lines is $v'$, and so $v'$ is the origin. In either case, this contradicts the fact that $(\alpha,v')>0$.
\end{proof}

\begin{corollary}\label{xs0nonempty}
    If $x$ does not share a panel with $H_{r_1}$ and $0\in D_R(x)$, then $xs_0\cap H_{r_0}\neq \varnothing$. 
\end{corollary}
\begin{proof}
    Let $v$ be the vertex of $x$ in the orbit of the origin. Then $v$ lies on $H_{r_1}$. The reflection over $H_{xs_0x^{-1}}$ must move $v$ off of $H_{r_1}$, otherwise $x$ would share a panel with $H_{r_1}$. But then $v$ is in the orbit of the origin, so there is a hyperplane parallel to $H_{r_0}$ and $H_{r_1}$ passing through $v$. Furthermore, as the $s_0$ panel of $x$ is fixed under reflection over $H_{xs_0x^{-1}}$, and the $s_0$ panel of $x$ does not lie on $H_{r_0}$ or $H_{r_1}$, $v$ must still lie between $H_{r_0}$ and $H_{r_1}$. So, as $H_{r_0}$ and $H_{r_1}$ are adjacent, $v$ must lie on $H_{r_0}$.
\end{proof}

\begin{lemma}\label{nopanel}
    If $x$ does not share a panel with $H_{r_0}$, then $D_R(r_0x)\cap D_R(x)\cap D_R(r_1x)\cap D_R(r_2r_1x)$ is nonempty. 
\end{lemma}
\begin{proof}
         By Proposition \ref{nonemptyor}, either $x$, $r_1x$ and $r_2r_1x$ have a shared decreasing panel type, or there is a generator $s_j$ such that $r_1xs_j=r_2r_1x$. Suppose they do not have a shared decreasing panel type. But then, as $r_2r_1=r_1r_0$, we have $r_1xs_j=r_1r_0x$. Therefore, $xs_j=r_0x$, and so $x$ shares a panel with $H_{r_0}$. This is a contradiction. Hence, $x$, $r_1x$ and $r_2r_1x$ have a shared decreasing panel type, say $s_i$. Now, by Lemma \ref{sharedxr1x}, $i\in D_R(r_0x)$.
\end{proof} 

We next prove that we can find a sequence of generators that allows us to assume that either $x$ shares a panel with $H_{r_1}$, or $x$ shares a simplex with $H_{r_0}$.

\begin{proposition}\label{sequence}
    There is a sequence of generators $s_{i_1},\hdots, s_{i_n}$ such that
    \[i_k\in D_R(r_0xs_{i_1}\cdots s_{i_{k-1}})\cap D_R(xs_{i_1}\cdots s_{i_{k-1}})\cap D_R(r_1xs_{i_1}\cdots s_{i_{k-1}})\cap D_R(r_2r_1xs_{i_1}\cdots s_{i_{k-1}}),\]
    and $r_1xs_{i_1}\cdots s_{i_{n}}\cap H_{r_0}\neq \varnothing$.
    \end{proposition}

\begin{proof}
    If $x\cap H_{r_0}\neq \varnothing$, then we are done. So assume that $x\cap H_{r_0}= \varnothing$. Then let $v$ be the vertex of $x$ that is in the orbit of the origin. Now $v$ must lie on $H_{r_1}$. Now let $x'\in W_0^v$ be the minimal element. We proceed by induction on $\ell(x)-\ell(x')$. 

    First consider when $x'$ shares a panel with $H_{r_0}$. Here, our base case is when $\ell(x)-\ell(x')=1$, so $x$ is preminimal. Now, by Lemma \ref{preminimal}, $0\in D_R(x)$. So, by Lemma \ref{s0descent}, $0\in D_R(r_0x)\cap D_R(r_1x)\cap D_R(r_2r_1x)$. Now $x$ does not share a panel with $H_{r_0}$, so $xs_0\cap H_{r_0}\neq \varnothing$. So we choose the sequence of $s_0$, and we are done. Now consider when $x'$ does not share a panel with $H_
    {r_0}$. Here, our base case is $\ell(x)-\ell(x')=0$, and so $x=x'$. Then $0\in D_R(x)$, and the same argument follows to give the sequence of $s_0$. 

    Now let us assume the claim is true for $\ell(x)-\ell(x')<n$, $n>1$. Let $x$ be such that $\ell(x)-\ell(x')=n$. Now there is a shared common decreasing panel type, say $s_i$. Now, if $s_i=s_0$, we are done by Corollary \ref{xs0nonempty}. Otherwise, $xs_i\in xW_0$ with $\ell(xs_i)-\ell(x')=n-1$, and so by induction there is a sequence $s_{i_1},\hdots, s_{i_n}$ such that
    \[i_k\in D_R(r_0xs_is_{i_1}\cdots s_{i_{k-1}})\cap  D_R(xs_is_{i_1}\cdots s_{i_{k-1}})\cap D_R(r_1xs_is_{i_1}\cdots s_{i_{k-1}})\cap D_R(r_2r_1xs_is_{i_1}\cdots s_{i_{k-1}}),\]
    and $r_1xs_is_{i_1}\cdots s_{i_{n}}\cap H_{r_0}\neq \varnothing$.
    So we just take the sequence $s_i,s_{i_1},\hdots, s_{i_n}$. 

    Now the last case to consider is when $x$ shares a panel with $H_{r_1}$ and $x=x'$. Now $x$, $r_1x$, and $r_2r_1x$ must share the decreasing panel type $s_0$. Now if $xs_0\cap H_{r_0}\neq \varnothing$, then we are done. Otherwise, $xs_0$ is not minimal in $xs_0W_0$ as $0\notin D_R(xs_0)$. So, as shown above, there is a sequence $s_{i_1},\hdots, s_{i_n}$ such that
    \[i_k\in D_R(r_0xs_0s_{i_1}\cdots s_{i_{k-1}})\cap  D_R(xs_0s_{i_1}\cdots s_{i_{k-1}})\cap D_R(r_1xs_0s_{i_1}\cdots s_{i_{k-1}})\cap D_R(r_2r_1xs_0s_{i_1}\cdots s_{i_{k-1}}),\]
    and $r_1xs_0s_{i_1}\cdots s_{i_{n}}\cap H_{r_0}\neq \varnothing$.
    So we just take the sequence $s_0,s_{i_1},\hdots, s_{i_n}$. 
    \end{proof}
We are now left to show that, if $x\cap H_{r_0}\neq \varnothing$, then $\ell(r_0x)=\ell(x)-1$. Then, by induction, we will be able to conclude that the statement holds also when $x\cap H_{r_0}= \varnothing$.
\begin{lemma}\label{minimalHinf}
    If $x$ is minimal in $xW_0\cap H_{r_0}^\infty$, then $x$ shares a panel with $H_{r_0}$. \end{lemma}
\begin{proof}
    If $x$ does not share a panel with $H_{r_0}$ then both panels of $x$ not of type $s_0$ must be increasing, as otherwise it is not minimal in $xW_0\cap H_{r_0}^\infty$. But then $x$ would be minimal in $xW_0$, contradicting the fact that $x\in H_{r_0}^\infty$. 
\end{proof}

\begin{proposition}\label{Hr0nonempty}
    If $x\cap H_{r_0}$ is nonempty and $\ell(r_0x)<\ell(x)$, then $\ell(r_0x)=\ell(x)-1$.  
\end{proposition}
\begin{proof}
    Let $x'$ be minimal in $xW_0\cap H_{r_1}^\infty$. Proceed by induction on $\ell(x)-\ell(x')$. If $x=x'$, then, by Lemma \ref{minimalHinf}, $x$ shares a panel with $H_{r_1}$, so clearly $\ell(r_0x)=\ell(x)-1$.

    Now suppose it is true for $x$ such that $\ell(x)-\ell(x')<n$ for $n>0$. Let $x$ be such that $\ell(x)-\ell(x')=n$. Now if $x$ shares a panel with $H_{r_0}$, the statement is obvious. If not, $r_0x$, $x$, $r_1x$ and $r_2r_1x$ share a common decreasing panel type $s_i\neq s_0$, by Lemma \ref{nopanel}. Now, $\ell(r_2r_1xs_i)=\ell(r_2r_1x)-1=\ell(r_1x)=\ell(r_1xs_i)+1$, and $\ell(r_1xs_i)=\ell(r_1x)-1=\ell(x)=\ell(xs_i)+1$, so by induction $\ell(r_0xs_i)=(\ell(xs_i)-1$. Now, $i\in D_R(r_0x)$, so $\ell(r_0x)=\ell(x)-1$. 
\end{proof}

We are now able to prove Theorem \ref{pm1}.

\begin{proof}
    If $\ell(r_0x)>\ell(x)$, or $x\cap H_{r_0}\neq \varnothing$, then we are done. So now suppose $\ell(r_0x)<\ell(x)$ and $x\cap H_{r_0}=\varnothing$. Now, by Lemma \ref{sequence} there is a sequence of generators $s_{i_1},\hdots, s_{i_n}$ such that
    \[i_k\in D_R(r_0xs_{i_1}\cdots i_{k-1})\cap  D_R(xs_{i_1}\cdots s_{i_{k-1}})\cap D_R(r_1xs_{i_1}\cdots s_{i_{k-1}})\cap D_R(r_2r_1xs_{i_1}\cdots s_{i_{k-1}}),\]
    and $r_1xs_{i_1}\cdots s_{i_{n}}\cap H_{r_0}\neq \varnothing$. Now $\ell(r_2r_1xs_{i_1}\cdots s_{i_{n}})=\ell(r_2r_1x)+n=\ell(r_1x)+n+1=\ell(r_1xs_{i_1}\cdots s_{i_{n}})+1$, and $\ell(r_1xs_{i_1}\cdots s_{i_{n}})=\ell(r_1x)+n=\ell(x)+n+1=\ell(xs_{i_1}\cdots s_{i_{n}})+1$, so, by Lemma \ref{Hr0nonempty}, $\ell(r_0xs_{i_1}\cdots s_{i_{n}})=\ell(xs_{i_1}\cdots s_{i_{n}})-1$. Therefore, $\ell(r_0x)=\ell(x)-1$. 
\end{proof}

\subsection{Proof of Theorem \ref{theorem}}

We can now use these results to prove Theorem \ref{theorem} inductively. We start by defining the Property $(\dagger)$ on a tuple $((r_1,\hdots,r_n),w,s_i)$. 

\begin{definition}
    A tuple $((r_1,\hdots,r_n),w,s_i)$  satisfies \ix{Property} $\boldsymbol{(\dagger)}$ if
    \begin{enumerate}
        \item $r_1,\hdots, r_n$ are a set of $n\geq 2$ parallel reflections such that $H_{r_i}$ is adjacent to $H_{r_{i-1}}$ and $H_{r_{i+1}}$;
        \item $w\in W$ lies between $H_{r_0}$ and $H_{r_1}$;
        \item $i\in D_R(w)$,
    \[\ell(r_t\cdots r_1ws_i)=\ell(r_{t-1}\cdots r_1 ws_i)+1,\] 
    and $r_{t}\cdots r_1ws_i\neq r_{t-1}\cdots r_1w$, for all $1\leq t\leq n$.
    \end{enumerate}
\end{definition}

\begin{remark}
    We will show that, if $(\dagger)$ holds on $((r_1,\hdots,r_n),w,s_i)$, then $r_n\cdots r_1ws_i\in \nPSh{w}$. Note that, if $n=1$, our statement trivially holds as $\ell(r_1ws_i)=\ell(w)$, and $w\neq r_1ws_i$. 
\end{remark}

We now prove the inductive steps in every case such that the identity is bounded by $H_{r_0}$ and $H_{r_1}$. We restrict to this condition, as when the identity is not bounded by $H_{r_0}$ and $H_{r_1}$, Theorem \ref{pm1} tells us that $\ell(r_0ws_i)=\ell(ws_i)-1$. In this case, we form an inductive argument with the tuple $((r_0,r_1,\hdots,r_n),r_0w,s_i)$. 

\begin{lemma}\label{daggerholds}
    Suppose the identity is bounded by $H_{r_0}$ and $H_{r_1}$. If $(\dagger)$ holds for $((r_1,\hdots,r_n), w,s_i)$ and $k\in D_R(r_t\cdots r_1 ws_i)$ for all $0\leq t\leq n$, then $(\dagger)$ holds for $((r_1,\cdots,r_n),ws_i,s_k).$ 
\end{lemma}

\begin{proof}
    First note that $ws_i$ must also be bounded by $H_{r_0}$ and $H_{r_1}$, otherwise $i\notin D_R(w)$. Then $k\in D_R(ws_i)$, and 
    \begin{align*}
        \ell(r_t\cdots r_1ws_is_k)&=\ell(r_t\cdots r_1ws_i)-1\\&=\ell(r_{t-1}\cdots r_1ws_i)\\&=\ell(r_{t-1}\cdots r_1ws_is_k)+1.
    \end{align*}
    Also note that $r_t\cdots r_1ws_is_k\neq r_{t-1}\cdots r_1 ws_i$ for all $1\leq t\leq n$, otherwise $k\notin D_R(r_{t-1}\cdots r_1 ws_i)\cap D_R(r_t\cdots r_1ws_i)$. So $(\dagger)$ holds for $((r_1,\hdots,r_n),ws_i,s_k)$
\end{proof}

First we consider the inductive step of the proof for the case that $|D_R(w)|=1$. 
\begin{proposition}
    Suppose that, for any tuple $((r_1,\hdots,r_n),w,s_i)$ with $\ell(w)<m$ and such that $(\dagger)$ holds, then $r_n\cdots r_1ws_i\in \nPSh{w}$. Suppose that $H_{r_0}$ and $H_{r_1}$ bound the identity, $\ell(w)=m$, $|D_R(w)|=1$ and $(\dagger)$ holds on $((r_1,\hdots,r_n),w,s_i)$. Then $r_n\cdots r_1ws_i\in \nPSh{w}$.
\end{proposition}

\begin{proof}
    First suppose that $\{r_t\cdots r_1ws_i\mid 0\leq t\leq n\}$ all share a decreasing panel type, say $s_k\neq s_i$. Now , by Lemma \ref{daggerholds}, $(\dagger)$ holds for $((r_1,\hdots,r_n),ws_i,s_k)$, and $\ell(ws_i)<\ell(w)$. So $ws_i\not\leq r_n\cdots r_1ws_is_k.$
    Now suppose $w\leq r_n\cdots r_1ws_i$. Then, as $k\in D_R(r_n\cdots r_1 ws_i)$, $k\notin D_R(w)$, $ws_i\leq w\leq r_n\cdots r_1ws_is_k$. But this is a contradiction, so $r_n\cdots r_1ws_i\in \nPSh{w}$. 

    Now suppose that $\{r_t\cdots r_1ws_i\mid 0\leq t\leq n\}$ do not all share a common decreasing panel type. Then, by Lemma \ref{nonemptyor2}, $\{r_t\cdots r_1ws_i\mid 0\leq t\leq n-1\}$ share a common decreasing panel type, say $s_k\neq s_i$, and there is an $s_j$ such that $r_{n-1}\cdots r_1ws_is_j=r_n\cdots r_1 ws_i$. Now, by Lemma \ref{daggerholds}, $(\dagger)$ holds on $((r_1,\hdots,r_{n-1}),ws_i,s_k)$, and so 
    \[ws_i\not\leq r_{n-1}\cdots r_1ws_is_k.\]
    Suppose  $w\leq r_n\cdots r_1ws_i$. By Corollary \ref{assume}, $w\leq r_{n-1}\cdots r_1ws_i$. But then $k\in D_R(r_{n-1}\cdots r_1 ws_i)$, $k\notin D_R(w)$, so $ws_i\leq w\leq r_{n-1}\cdots r_1ws_is_k$. But this is a contradiction, so $r_n\cdots r_1ws_i\in \nPSh{w}$. 
 \end{proof}

 We next show that, if $|D_R(w)|=2$, the $D_R(w)=\{0,i\}$ for some $i\in \{1,2\}$. 

 \begin{lemma}\label{notmaximal}
     If $w\in W$ and the identity alcove are both bounded by two parallel adjacent hyperplanes $H_{r_0}$ and $H_{r_1}$, then $w$ cannot be maximal in $wW_0$. 
 \end{lemma}

 \begin{proof}
     Let $v$ be the vertex of $w$ in the orbit of the origin. Then $v$ must lie on $H_{r_0}$ or $H_{r_1}$. Say it lies on $H_{r_i}$. Then $\ell(r_iw)>\ell(w)$, as $H_{r_i}$ separates $r_iw$ from the identity. But now $r_iv=v$, so $r_iwW_0=wW_0$. Therefore, $r_iw\in wW_0$, and so $w$ is not maximal in $wW_0$. 
 \end{proof}

 \begin{corollary}
     If $w\in W$ and the identity alcove are both bounded by two parallel adjacent hyperplanes $H_{r_0}$ and $H_{r_1}$, and $|D_R(w)|=2$, then $0\in D_R(w)$. 
 \end{corollary}
 \begin{proof}
     If $|D_R(w)|=2$, but $0\notin D_R(w)$, then $D_R(w)=\{1,2\}$, and so is maximal in $wW_0$, contradicting Lemma \ref{notmaximal}. 
 \end{proof}

We now can show that, in these conditions, the only time $ws_i$ is not minimal in $wW_0$ is when the element $s_0s_i$ has order two. We can further show that $s_0s_i$ never has order bigger than three. 

\begin{proposition}\label{minimal}
    Suppose $w\in W$ and the identity alcove are both bounded by two parallel adjacent hyperplanes $H_{r_0}$ and $H_{r_1}$, and suppose that $D_R(w)=\{0,i\}$. If $(s_0s_i)^2\neq 1$, then $ws_i$ is minimal in $wW_0$. 
\end{proposition}

\begin{proof}
     Note that $0\in D_R(w)$, so as $ws_i\in wW_0$ and $ws_i\leq w$, $0\in D_R(ws_i)$. Suppose $ws_i$ is not minimal in $wW_0$. Then $D_R(ws_i)=\{0,j\}$, $j\notin \{0,i\}$. Now $w$ is maximal in $wW_{\{s_0,s_i\}}$, and as $(s_0s_i)^2\neq 1$, we have that $i\in D_R(ws_is_0)$. Now $ws_i$ is maximal in $ws_iW_{\{s_0,s_j\}}$, so $j\in D_R(ws_is_0)$. Then $i,j\in D_R(ws_is_0)$, so $ws_is_0$ is maximal in $ws_is_0W_0$. But $w$ and the identity are bounded by $H_{r_0}$ and $H_{r_1}$, and $i\in D_R(w)$, $0\in D_R(ws_i)$, so $ws_is_0$ is bounded by $H_{r_0}$ and $H_{r_1}$. Hence, by Lemma \ref{notmaximal}, it cannot be maximal in its $W_0$ coset, and hence we have a contradiction. Therefore, $ws_i$ must be minimal in $ws_iW_0$. 
    \end{proof}

\begin{lemma}\label{notorder4}
    Suppose $w\in W$ and the identity alcove are both bounded by two parallel adjacent hyperplanes $H_{r_0}$ and $H_{r_1}$, and suppose that $D_R(w)=\{0,i\}$. Then the element $s_0s_i$ does not have order 4. 
\end{lemma}

\begin{proof}
    Suppose $s_0s_i$ has order 4. Then we must have a $\tilde{B}_2$ Coxeter complex. Hence, $s_1s_2$ must also have order 4. Therefore, any $\{s_0,s_i\}$ vertex has all directions of hyperplanes passing through it. So the $\{s_0,s_i\}$ vertex $v$ of $w$ must lie on $H_{r_0}$ or $H_{r_1}$. Say it lies on $H_{r_t}$. Then $\ell(r_tw)>\ell(w)$, but $r_t$ fixes $v$. So $r_tw\in wW_{\{s_0,s_i\}}$. Therefore, $w$ cannot be maximal in $wW_{\{s_0,s_i\}}$, and hence $D_R(w)\neq \{0,i\}$. 
\end{proof}

We are now ready to prove Theorem \ref{theorem}, by an induction argument on the length of $w$. Here, we restate the theorem in terms of the Property $(\dagger)$. 
\begin{theorem}\label{daggermain}
        If $(\dagger)$ holds on $((r_1,\hdots,r_n),w,s_i)$, then $r_n\cdots r_1ws_i\in \nPSh{w}$. 
\end{theorem}

\begin{proof}
    We proceed by induction on $\ell(w)$. 

   If $\ell(w)=0$, then $w=1$, and $(\dagger)$ can never hold as $D_R(w)=\varnothing$. If $\ell(w)=1$, $w=s_i$ for some $s_i$, and $D_R(w)=\{i\}$. So assume that $(\dagger)$ holds for $((r_1,\hdots,r_n),w,s_i)$. Now $\ell(r_1ws_i)=\ell(ws_i)+1=1$. So $r_1ws_i=s_j$. Note that $j\neq i$. Now if $n>1$, $H_{r_1}$ and $H_{r_2}$ do not bound $w$ or $ws_i=1$. So, as $H_{r_0}$ and $H_{r_1}$ bound $w$ and $ws_i=1$, no vertex of the identity lies on $H_{r_2}$. So the reflection $r_2$ does not fix any vertex of the origin. Therefore, $r_2r_1ws_i=r_2s_j$ cannot lie in any maximal parabolic subgroup. But any length two element must lie in a maximal parabolic subgroup. Therefore, if $(\dagger)$ holds for $((r_1,\hdots,r_n),w,s_i)$, $n=1$, which is a contradiction.  

    Now let us assume that, if $(\dagger)$ holds on $((r_1,\hdots,r_n),w,s_i)$ for $\ell(w)<m$, then $r_n\cdots r_1ws_i\in \nPSh{w}$. 

    Now assume that $\ell(w)=m$. Note that, by Corollary \ref{pm1}, $\ell(r_0w)=\ell(w)\pm 1$. First suppose that $\ell(r_0w)=\ell(w)- 1$. Then it is easy to see that $(\dagger)$ holds on $((r_0,r_1,\hdots,r_n),r_0w,s_i)$. So, by induction, $r_n\cdots r_1r_0r_0ws_i\in\nPSh{r_0w}\subseteq \nPSh{w}$. So we can assume that $\ell(r_0w)>\ell(w)$, and hence $H_{r_0}$ and $H_{r_1}$ bound the identity alcove. 

    Now we first assume that $D_R(w)=\{i\}$. Then, as $s_i$ is a decreasing panel of $w$ and $H_{r_0}$ and $H_{r_1}$ bound the identity alcove, $H_{r_0}$ and $H_{r_1}$ must bound $ws_i$. Now first suppose that $ws_i, r_1ws_i,\hdots, r_{n}\cdots r_1ws_i$ share a common decreasing panel $s_k\neq s_i$. So $k\notin D_R(w)$. Now if $w\leq r_n\cdots r_1ws_i$, then $ws_i\leq w\leq r_n\cdots r_1ws_is_k$. But let us consider  $((r_1,\hdots,r_n),ws_i,s_k)$. Now
    \[\ell(r_t\cdots r_1ws_is_k)=\ell(r_t\cdots r_1ws_i)-1=\ell(r_{t-1}\cdots r_1ws_i)=\ell(r_{t-1}\cdots r_1ws_is_k)+1,\]
    and $r_t\cdots r_1ws_is_k\neq r_{t-1}\cdots r_1ws_i$ for any $0\leq t\leq n$, otherwise $ws_is_k$ would not be bounded by $H_{r_0}$ and $H_{r_1}$, and then $s_k$ would not be a decreasing panel of $ws_i$. So $(\dagger)$ holds for $((r_1,\hdots,r_n),ws_i,s_k)$. So, by induction, $ws_i\not\leq r_n\cdots r_1ws_is_k$, and so if $w\leq r_n\cdots r_1ws_i$ we get a contradiction. Hence, $r_n\cdots r_1ws_i\in \nPSh{w}$.

     Now assume that $ws_i, r_1ws_i,\hdots, r_{n-1}\cdots r_1ws_i$ share a common decreasing panel $s_k\neq s_i$, and that there is a panel type $s_j$ such that $r_{n-1}\cdots r_1ws_is_j=r_n\cdots r_1ws_i$. If $w\leq r_n\cdots r_1 ws_i$, by Corollary \ref{assume}, $w\leq r_{n-1}\cdots r_1ws_i$. Then $k\notin D_R(w)$, $k\in D_R(r_{n-1}\cdots r_1 ws_i$, so $ws_i\leq w\leq r_{n-1}\cdots r_1ws_is_k$. Now $(\dagger)$ holds for $((r_1,\dots,r_{n-1}),ws_i,s_k)$, so $ws_i\not\leq r_{n-1}\cdots r_1ws_is_k$, which is a contradiction. So $r_n\cdots r_1ws_i\in \nPSh{w}$.

     Now assume that $|D_R(w)|=2$. Now $w$ is not maximal in $wW_0$, so $0\in D_R(w)$. So $D_R(w)=\{0,i\}$ for some $i\in \{1,2\}$. Assume that $(\dagger)$ holds for $((r_1,\hdots,r_n),w,s_0)$. First assume that $(s_0s_i)^2=1$. Now $w$ is maximal in $wW_{\{s_0,s_i\}}$. By the same argument above, we only need to consider the case that all $ws_0, r_1ws_0,\hdots, r_{n}\cdots r_1ws_0$ share a common decreasing panel type $s_k\neq s_0$. If $s_k\neq s_i$, we are done by Lemma \ref{daggerholds}. So suppose they all share $s_i$ as a decreasing panel type. Now $0\in D_R(r_t\cdots r_1 w)$ and $i\in D_R(r_t\cdots r_1ws_0)$, so, as $(s_0s_i)^2=1$,  $r_t\cdots r_1w$ is maximal in its $W_{\{s_0,s_i\}}$ coset. Therefore, $i\in D_R(r_t\cdots r_1w)$. Therefore, $(\dagger)$ holds for $((r_1,\hdots,r_n),ws_i,s_0)$. So by induction $ws_i\not\leq r_n\cdots r_1ws_is_0$. Now if $w\leq r_n\cdots r_1 ws_0$, then $ws_i\leq r_n\cdots r_1ws_0s_i=r_n\cdots r_1ws_is_0$, which is a contradiction. The same argument holds if  $(\dagger)$ holds for $((r_1,\hdots,r_n),w,s_i)$.

     Now suppose that $(s_0s_i)^3=1$. By Lemmas \ref{minimal} and \ref{preminimal}, $w$ must be preminimal, i.e.\ $ws_i$ is minimal in $wW_0=ws_iW_0$. First suppose that  $(\dagger)$ holds for $((r_1,\hdots,r_n),w,s_0)$. We can assume that  all $ws_0, r_1ws_0,\hdots, r_{n}\cdots r_1ws_0$ share a common decreasing panel type $s_k\neq s_0$. If $s_k\neq s_i$, we are done by Lemma \ref{daggerholds}. So assume that $s_k=s_i$. Now, by Corollary \ref{nonemptyor2}, either all $ws_0s_i, r_1ws_0s_i,\hdots, r_{n}\cdots r_1ws_0s_i$ share a common decreasing panel type, or all but $r_n\cdots r_1ws_0s_i$ share a decreasing panel type. 
     
     Case 1: All $ws_0s_i, r_1ws_0s_i,\hdots, r_{n-1}\cdots r_1ws_0s_i$ share a common decreasing panel type. This type cannot be $s_i$. 
    
     Case 1a: First suppose it is $s_k\neq s_0$. Then $(\dagger)$ holds for $((r_1,\hdots,r_n),ws_0s_i,s_k)$. So, by induction, $ws_0s_i\not\leq r_n\cdots r_1ws_0s_is_k$. But now, if $w\leq r_n\cdots r_1ws_0$, $ws_i\leq r_n\cdots r_1ws_0s_i$. Now $ws_i$ is minimal, so $k\notin D_R(ws_i)$. So $ws_0s_i\leq ws_i\leq r_n\cdots r_1 ws_0s_is_k$, which is a contradiction. 
    
    Case 1b: Now suppose they all share $s_0$ as a common decreasing panel type. Now, $0\in D_R(r_t \cdots r_1 w)$, $i \in D_R(r_t \cdots r_1 ws_0)$, and $0\in D_R(r_t \cdots r_1 ws_0s_i).$ So, as $(s_0s_i)^3=1$, $w$ must be maximal in $wW_{\{s_0,s_i\}}$. Then $i\in D_R(r_t \cdots r_1 w)$, $0\in D_R(r_t \cdots r_1 ws_i)$, and $i\in D_R(r_t \cdots r_1 ws_is_0)$. So  $(\dagger)$ holds for $((r_1,\hdots,r_n),ws_is_0,s_i)$. Therefore, $ws_is_0\not\leq r_n\cdots r_1ws_is_0s_i$. But if $w\leq r_n\cdots r_1ws_0$, then $ws_is_0\leq r_n\cdots r_1 ws_0s_is_0=r_n\cdots r_1 ws_is_0s_i$, which is a contradiction. 

    Case 2: All $ws_0s_i, r_1ws_0s_i,\hdots, r_{n-1}\cdots r_1ws_0s_i$ share a common decreasing panel type, and there is an $s_j$ such that $r_{n-1
    }\cdots r_1 ws_is_j=r_n\cdots r_1ws_i$. By applying the same argument above to $((r_1,\hdots,r_{n-1}),w,s_0)$, we can conclude that $w\not\leq r_{n-1}\cdots r_1ws_0$. But then, by Corollary \ref{assume}, $w\not\leq r_n\cdots r_1ws_0$.

    A very similar argument holds if $(\dagger)$ holds for $((r_1,\hdots,r_{n}),w,s_i)$.

    Now, by Lemma \ref{notorder4}, $s_0s_i$ cannot have order 4 if $D_R(w)=\{0,i\}$. So these cases are exhaustive, and our induction is complete.  
\end{proof}

Now the next corollary follows easily, by combining Theorem \ref{daggermain} with Proposition \ref{rn...r1winHid}.

\begin{corollary}\label{cor}
    If $(\dagger)$ holds on $((r_1,\hdots,r_n),w,s_i)$, then $r_n\cdots r_1ws_i$ is on the boundary of $\nPSh{w}$. 
\end{corollary}

\begin{example}
    Here we see an example of the application this theorem. We are considering the annex of the alcove $x=s_0s_1s_2s_0s_1s_0s_2$ from Figure \ref{alcovepic}.
    \begin{figure}[htbp!]
        \centering
        \includegraphics[width=0.3\linewidth]{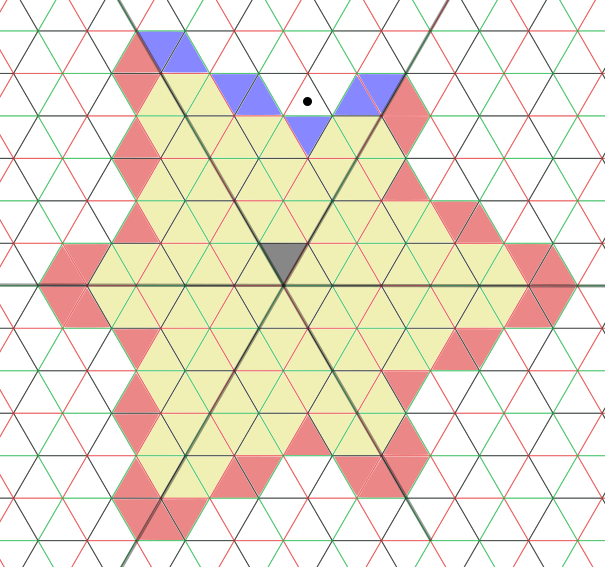}
        \caption{ The purple alcoves are in the annex, by Theorem \ref{theorem}. Then, by Corollary \ref{cor}, we know these elements are on the boundary of the annex. Then, by Proposition \ref{symmetry2}, as our element is in the fundamental chamber, we know by symmetry that the red alcoves must also be on the boundary of the annex. This completely defines the annex. }
        \label{fig:placeholder}
    \end{figure}
    
\end{example}

\bibliographystyle{acm}

\bibliography{references}

\end{document}